\documentclass[12pt]{article}
\usepackage{amsfonts, amsmath, amssymb, latexsym, eucal, amscd} 
\usepackage[dvips]{pict2e}
\usepackage{amsthm}
\usepackage{amscd}
\usepackage[dvipdfmx]{graphics}

\usepackage{cite}

\newtheorem{definition}{Definition}[section]
\newtheorem{theorem}[definition]{Theorem}
\newtheorem{lemma}[definition]{Lemma}
\newtheorem{corollary}[definition]{Corollary}

\newtheorem{example}[definition]{Example}

\newtheorem{problem}[definition]{Problem}

\newtheorem{proposition}[definition]{Proposition}

\begin{document}

\title{\bf
Twisting finite-dimensional modules 
\\
 for the $q$-Onsager algebra $\mathcal O_q$ via the
 \\ Lusztig automorphism
}
\author{
Paul Terwilliger}
\date{}

\maketitle
\begin{abstract}
The $q$-Onsager algebra $\mathcal O_q$ is defined by two generators $A$, $A^*$ and two relations, called the
$q$-Dolan/Grady relations. Recently P. Baseilhac and S. Kolb found an automorphism $L$ of $\mathcal O_q$,
that fixes $A$ and sends $A^*$ to a linear combination of $A^*$, $A^2A^*$, $AA^*A$, $A^*A^2$.
Let $V$ denote an irreducible $\mathcal O_q$-module of finite dimension at least two, on which each of $A$, $A^*$ 
is diagonalizable. It is known that $A$, $A^*$ act on $V$ as a tridiagonal pair of $q$-Racah type, giving access to four familiar 
 elements $K$, $B$, $K^\downarrow$, $B^\downarrow$ in ${\rm End}(V)$
that are used to compare the eigenspace decompositions for $A$, $A^*$ on $V$.
 We display  an invertible $H \in {\rm End}(V)$ such that $L(X)=H^{-1} X H$ on $V$
for all $X \in \mathcal O_q$. We describe what happens when one of $K$, $B$, $K^\downarrow$,
$B^\downarrow$ is conjugated by $H$. For example $H^{-1}KH=a^{-1}A-a^{-2}K^{-1}$ where  $a$
is a certain scalar that is used to describe the eigenvalues of $A$ on $V$. We use the conjugation results to compare 
the eigenspace decompositions for  $A$, $A^*$, $L^{\pm 1}(A^*)$ on $V$. In this 
comparison we use the notion of an equitable triple; this is a 3-tuple of elements in ${\rm End}(V)$ such that any two
satisfy a $q$-Weyl relation. Our comparison involves eight equitable triples. One of them is $a A - a^2 K$, $M^{-1}$, $K$ where
$M= (a K-a^{-1} B)(a-a^{-1})^{-1}$. The map $M$ appears in earlier work of S. Bockting-Conrad concerning the
double lowering operator $\psi$ of a tridiagonal pair.
\bigskip

 \end{abstract}

\section{Introduction}
In this paper we consider a certain kind of  module for the $q$-Onsager algebra $\mathcal O_q$.
Before getting into detail, we briefly set some notation.
Recall the natural numbers $\mathbb N = \lbrace 0,1,2,\ldots\rbrace$ and integers
$\mathbb Z = \lbrace 0, \pm 1, \pm 2,\ldots \rbrace$.
Let $\mathbb F$ denote a field. Every vector space discussed in this paper is over $\mathbb F$.
 Every algebra discussed in this paper is associative, over $\mathbb F$, and has a multiplicative identity.
Fix a nonzero $q \in \mathbb F$ that is not a root of unity. Recall the notation
\begin{align*}
\lbrack n \rbrack_q = \frac{q^n-q^{-n}}{q-q^{-1}} \qquad \qquad n \in \mathbb Z.
\end{align*}
For elements $X$, $Y$ in any algebra, their
commutator and $q$-commutator are given by
\begin{align*}
\lbrack X, Y\rbrack = XY-YX, \qquad \qquad
\lbrack X, Y\rbrack_q = q XY - q^{-1} YX.
\end{align*}
Note that
\begin{align*}
\lbrack X, \lbrack X, \lbrack X, Y\rbrack_q \rbrack_{q^{-1}} \rbrack =  X^3 Y - \lbrack 3 \rbrack_q X^2 Y X + \lbrack 3 \rbrack_q X Y X^2 - Y X^3.
\end{align*}

\begin{definition}
\label{def:qOns}  \rm
(See \cite[Section~2]{basDeform},
  \cite[Definition~3.9]{qSerre}.)
The {\it $q$-Onsager algebra} ${\mathcal O}_q$ is defined by generators $A$, $A^*$ and relations
\begin{align}
\label{eq:DG1}
&
\lbrack A, \lbrack A, \lbrack A, A^*\rbrack_q \rbrack_{q^{-1}}\rbrack = (q^2-q^{-2})^2 \lbrack A^*, A\rbrack,
\\
\label{eq:DG2}
& 
\lbrack A^*, \lbrack A^*, \lbrack A^*, A \rbrack_q \rbrack_{q^{-1}} \rbrack = (q^2-q^{-2})^2 \lbrack A, A^*\rbrack.
\end{align}
The relations {\rm (\ref{eq:DG1})}, {\rm (\ref{eq:DG2})} are called the {\it $q$-Dolan/Grady relations}.
\end{definition}

\noindent The algebra $\mathcal O_q$ comes from algebraic graph theory,  or more precisely, the theory of $Q$-polynomial distance-regular graphs
 \cite{bannaiIto},
\cite{BCN}.
 For such a graph, the  adjacency matrix and any dual adjacency matrix satisfy
 two relations   \cite[Lemma~5.4]{tersub3} called the tridiagonal relations \cite[Definition~3.9]{qSerre}.
 If the graph has $q$-Racah type \cite[Section~1]{ItoTerwilliger1} then these tridiagonal relations become
  the $q$-Dolan/Grady relations after an appropriate normalization. For an overview of this topic see \cite{augmented}.
\medskip

\noindent The algebra $\mathcal O_q$ comes up in the theory of tridiagonal pairs. Roughly speaking, a  tridiagonal pair
is a pair of diagonalizable linear maps on a nonzero finite-dimensional vector space, that each act
in a block-tridiagonal fashion on 
the eigenspaces of the other one \cite[Definition~1.1]{TD00}. A finite-dimensional
irreducible $\mathcal O_q$-module on which the generators are diagonalizable is essentially the same thing
as a tridiagonal pair of $q$-Racah type  \cite[Theorem~3.10]{qSerre}. For more information on $\mathcal O_q$ and tridiagonal pairs, see
\cite{ItoTD, INT, tdqrac, augmented,  qSerre,
2lintrans,
madrid,
univTer,
pospart,
terLusztig,
qOnsUniv,
vidter}.
\medskip

\noindent The algebra $\mathcal O_q$ has applications
to quantum integrable models
\cite{bas2,
basDeform,
 bas8,
basXXZ,
BK05,
bas4,
  basKoi},
reflection equation algebras
\cite{basBel,basnc},
and coideal subalgebras
\cite{bc,kolb, kolb1}.
There is an algebra homomorphism from $\mathcal O_q$ into
the algebra $\square_q$ 
\cite[Proposition~5.6]{pospart},
and the universal Askey-Wilson algebra
\cite[Sections~9,10]{univTer}.
\medskip

\noindent We will be discussing automorphisms of $\mathcal O_q$. By an {\it automorphism} of $\mathcal O_q$
we mean an algebra isomorphism $\mathcal O_q \to \mathcal O_q$.  In
 \cite{BK} P. Baseilhac and S. Kolb
introduced the following automorphism of $\mathcal O_q$.
\begin{lemma} 
\label{lem:Lusztig} 
{\rm (See \cite[Section~2]{BK}.)}
There exists an automorphism $L$ of $\mathcal O_q$ such that
\begin{align}
L(A)=A, \qquad \qquad
L(A^*)=A^* + \frac{\lbrack A, \lbrack A, A^*\rbrack_q \rbrack }{(q-q^{-1})(q^2-q^{-2})}.
\label{eq:Lusztig}
\end{align}
The inverse automorphism $L^{-1}$ satisfies
\begin{align}
\label{eq:Linv}
L^{-1}(A)= A, \qquad \qquad
L^{-1}(A^*)= A^* + \frac{\lbrack A, \lbrack A, A^*\rbrack_{q^{-1}} \rbrack }{(q-q^{-1})(q^2-q^{-2})}.
\end{align}
\end{lemma}
\noindent The automorphism $L$ is roughly analogous to the Lusztig automorphism of
the quantum group $U_q({\widehat{sl}}_2)$. Motivated by this,
we call $L$ the {\it Lusztig automorphism of $\mathcal O_q$} \cite{terLusztig}.
\medskip

\noindent The contents of the present paper are summarized as follows.
Let $V$ denote a finite-dimensional irreducible $\mathcal O_q$-module on which each of $A$, $A^*$ is diagonalizable.
To avoid trivialities, we assume that $V$ has dimension at least 2. We mentioned below Definition 
\ref{def:qOns}
 that $A$, $A^*$ act on $V$
as a tridiagonal pair of $q$-Racah type. We give a detailed proof of this fact in order to set the stage. We then 
 display an invertible $H \in {\rm End}(V)$ such that $L(X) = H^{-1} XH$ on $V$ for all $X \in \mathcal O_q$.
The existence of $H$ means that the  $\mathcal O_q$-module $V$ is isomorphic to the $\mathcal O_q$-module
$V$ twisted via $L^{\pm 1}$. From another point of view, it means that the following act on $V$ as isomorphic tridiagonal pairs:
(i) $A$, $A^*$; (ii) $A$, $L(A^*)$; (iii) $A$, $L^{-1}(A^*)$.
We display some identities that are satisfied by the eigenvalues of $H$. These identities are obtained using the Chu/Vandermonde
summation formula for basic hypergeometric series. Using the identities we express $H^{\pm 1}$ as a polynomial
in $A$.
We then consider some elements $K$, $B$, $K^\downarrow$, $B^\downarrow$ in ${\rm End}(V)$. These elements
are familiar in the theory of tridiagonal pairs; they are used to compare the eigenspace decompositions for $A$, $A^*$ on $V$.
The element $K$ was introduced in \cite[Section~1.1]{augmented}.
The element $B$ is obtained from $K$
 by inverting the
ordering of the eigenspace decomposition for $A$ on $V$.
The elements $K^\downarrow$, $B^\downarrow$ are obtained from $K$, $B$ by inverting the ordering of the
eigenspace decomposition for $A^*$ on $V$.
In  \cite{bockting, bocktingQexp} S. Bockting-Conrad investigates $K$, $B$ in detail, as part of her work on
the double lowering operator $\psi$ \cite{bockting1}.  She obtains some relations involving
$K,B, A$ that she uses to turn $V$ into a module
for the quantum group $U_q(\mathfrak{sl}_2)$.
We describe what happens when one of $K$, $B$, $K^\downarrow$, $B^\downarrow$ is conjugated by $H$.
We show that 
\begin{align}
&H^{-1} B H = a A - a^{2} B^{-1},
\qquad \qquad \quad  \;\;
H^{-1} K H = a^{-1} A - a^{-2} K^{-1}, \label{eq:h1}
\\
&H^{-1} B^\downarrow H = a A - a^{2} (B^\downarrow)^{-1},
 \qquad \qquad
H^{-1} K^\downarrow  H = a^{-1} A - a^{-2} (K^\downarrow)^{-1}, \label{eq:h2}
\end{align}
\noindent where $a$ is a certain scalar that is used to describe the eigenvalues for $A$ on $V$; see Lemma
\ref{lem:form} below.
Reformulating the equations
(\ref{eq:h1}),
(\ref{eq:h2}) 
we obtain
\begin{align}
&H B^{-1} H^{-1} = a^{-1} A - a^{-2} B,
\qquad \qquad \quad  \;\;
H K^{-1} H^{-1}= a A - a^{2} K, \label{eq:h3}
\\
&H (B^\downarrow)^{-1} H^{-1} = a^{-1} A - a^{-2} B^\downarrow,
 \qquad \qquad
H (K^\downarrow)^{-1} H^{-1}= a A - a^{2} K^\downarrow. \label{eq:h4}
\end{align}
We use (\ref{eq:h1})--(\ref{eq:h4})  to compare the eigenspace decompositions of $A$, $A^*$, $L^{\pm 1}(A^*)$ on $V$. 
In this comparison we use the notion of an equitable triple; this is a 3-tuple
$X$, $Y$, $Z$ of invertible elements in ${\rm End}(V)$ such that
\begin{align*}
\frac{qXY-q^{-1} YX}{q-q^{-1} } = I,
\qquad \quad 
\frac{qYZ-q^{-1} ZY}{q-q^{-1} } = I,
\qquad \quad 
\frac{qZX-q^{-1} XZ}{q-q^{-1} } = I.
\end{align*}
Our comparison involves eight equitable triples; see Proposition \ref{thm:xyz}
below. One of the triples  is $a A - a^2 K$, $M^{-1}$, $K$ where
\begin{align*}
M= \frac{a K-a^{-1} B}{a-a^{-1}}.
\end{align*}
The element $M$ appears in the work of Bockting-Conrad mentioned above; see
\cite[Section~6]{bocktingQexp}. 
\medskip

\noindent
We describe how the eight equitable triples are related to the eigenspace decompositions for $A$, $A^*$, $L^{\pm 1}(A^*)$ on $V$.
We find it illuminating to make this description using diagrams; see
Theorems 
\ref{thm:Example9}--\ref{thm:Example7} below. These theorems are the main results of the paper.
\medskip

\noindent The paper is organized as follows. In Section 2 we prove that $A$, $A^*$ act on $V$ as a tridiagonal pair of $q$-Racah type.
In Section 3 we describe the Lusztig automorphism $L$ of $\mathcal O_q$, and display an invertible $H \in {\rm End}(V)$ such that
$L(X) = H^{-1} X H$ on $V$ for all $X \in \mathcal O_q$. We also obtain some identities involving the eigenvalues of $H$.
In Section 4 we introduce some diagrams that are used to compare the eigenspace decompositions for $A$, $A^*$, $L^{\pm 1}(A^*)$ on $V$.
In Section 5 we recall the elements $K$, $B$, $K^\downarrow$, $B^\downarrow$ and discuss their basic properties.
In Section 6 we describe what happens when one of $K$, $B$, $K^\downarrow$, $B^\downarrow$ is conjugated by $H$.
In Section 7 we discuss the notion of an equitable triple, and display eight equitable triples that involve
$K$, $B$, $K^\downarrow$, $B^\downarrow$. In Section 8 we describe how the eight equitable triples 
are related to 
the eigenspace decompositions for $A$, $A^*$, $L^{\pm 1}(A^*)$ on $V$. This description is made using diagrams.

\section{The $q$-Dolan/Grady relations and tridiagonal pairs}
\noindent  We now begin our formal argument. For the rest of this paper, $V$ denotes an irreducible $\mathcal O_q$-module
with finite dimension at least 2, on which each of $A$, $A^*$ is diagonalizable. In this section we show that $A$, $A^*$ act on $V$
as a tridiagonal pair of $q$-Racah type.
\medskip

\noindent 
We consider how $A$ and $A^*$ act on each other's eigenspaces. Let ${\rm End}(V)$ denote the algebra
consisting of the $\mathbb F$-linear maps from $V$ to $V$. Let $\mathcal D$ denote the set of 
eigenvalues for $A$ on $V$. By construction $\mathcal D$ is nonempty. We emphasize that the elements of $\mathcal D$ are mutually distinct.
For $\lambda \in \mathcal D$ define  $E_\lambda \in {\rm End}(V)$ to be the projection onto the $\lambda$-eigenspace for $A$ on $V$.
The following holds on $V$:
(i) $A E_\lambda = \lambda E_\lambda = E_\lambda A$ ($\lambda \in \mathcal D$); (ii)
$E_\lambda E_\mu = \delta_{\lambda, \mu} E_\lambda$ ($\lambda, \mu \in \mathcal D$);
(iii) $I = \sum_{\lambda \in \mathcal D} E_\lambda$.
For $\lambda, \mu \in \mathcal D$,
multiply each side of
(\ref{eq:DG1}) 
on the left by $E_\lambda$ and the right by $E_\mu$. After some routine simplification, we obtain
\begin{align*}
E_\lambda A^* E_\mu (\lambda - \mu)( q \lambda - q^{-1} \mu) (q^{-1} \lambda - q \mu)
= E_\lambda A^* E_\mu (\mu-\lambda) (q^2-q^{-2})^2.
\end{align*}
Therefore
\begin{align*} 
E_\lambda A^* E_\mu (\lambda - \mu)P(\lambda, \mu)=0,
\end{align*}
where
\begin{align}
P(\lambda, \mu) = \lambda^2- (q^2+q^{-2}) \lambda \mu + \mu^2 + (q^2-q^{-2})^2.
\label{eq:P}
\end{align}
Consequently
\begin{align}
\label{eq:choices}
E_\lambda A^* E_\mu= 0 \qquad {\mbox{\rm or}}  \qquad \lambda = \mu \qquad {\mbox{\rm or}} \qquad P(\lambda, \mu)=0.
\end{align}
Note that  $P(\lambda, \mu)= P(\mu, \lambda)$. 
Call $\lambda, \mu$ {\it adjacent} whenever $\lambda \not=\mu$ and $P(\lambda, \mu)=0$. The adjacency relation is symmetric.
\medskip

\noindent
We just defined a symmetric binary relation on $\mathcal D$, called adjacency.  This relation turns $\mathcal D$ into an undirected graph.
For the graph $\mathcal D$, each vertex is adjacent to at most two other vertices, since the polynomial $P$ has degree two. The graph $\mathcal D$ is
connected, because  $V$ is irreducible as an $\mathcal O_q$-module. By these comments, the graph $\mathcal D$ is either a path or a cycle.
Shortly we will show that a cycle cannot occur.
\begin{lemma} 
\label{lem:3v}
Referring to the graph $\mathcal D$,
let $\lambda $ denote a vertex that is adjacent to two distinct vertices $\mu, \nu$. Then
\begin{align*}
\mu - (q^2+q^{-2}) \lambda + \nu = 0.
\end{align*}
\end{lemma}
\begin{proof} The vertices $\lambda, \mu$ are adjacent, so $P(\lambda, \mu)=0$. The vertices $\lambda, \nu$
are adjacent, so $P(\lambda, \nu)=0$. By these comments and 
(\ref{eq:P}),
\begin{align*}
 0 = \frac{P(\lambda, \mu) - P(\lambda, \nu)}{\mu-\nu}
 = \mu - (q^2+q^{-2}) \lambda+ \nu.
 \end{align*}
\end{proof}

\begin{lemma} 
\label{lem:n2} We have $\vert \mathcal D \vert \geq 2$.
\end{lemma}
\begin{proof} By construction $\vert \mathcal D\vert \geq 1$. We assume that $\vert \mathcal D\vert =1$ and get a contradiction. The action of $A$ on $V$ is diagonalizable with a single eigenvalue.
So $A$ acts on $V$ as a scalar multiple of the identity. The action of $A^*$ on $V$ is diagonalizable, so there exists a nonzero $v \in V$ that is an eigenvector for $A^*$.
The subspace $W=\mathbb F v$ is invariant under $A$ and $A^*$. Therefore $W=V$ since $V$ is irreducible as an $\mathcal O_q$-module. Consequently $V$ has dimension one,
 a  contradiction. We have shown that $\vert \mathcal D\vert \geq 2$.
\end{proof}

\begin{definition}\rm Define $d=\vert \mathcal D\vert -1$. Let $\lbrace \theta_i \rbrace_{i=0}^d$
denote an ordering  of $\mathcal D$  such that $\theta_{i-1}, \theta_i$ are adjacent for $1 \leq i \leq d$. 
\end{definition}
 
 \begin{lemma} 
 \label{lem:form}
 The graph $\mathcal D$ is a path. Moreover, there exists $0 \not=a \in \mathbb F$ such that
 \begin{align*}
 \theta_i = a q^{d-2i} + a^{-1} q^{2i-d} \qquad \qquad (0 \leq i \leq d).
 \end{align*}
 \end{lemma}
 \begin{proof}  For notational convenience define $n = \vert \mathcal D \vert$. For the graph $\mathcal D$ let $e$ denote the number of (undirected) edges. So $e=d$ if
 $\mathcal D$ is a path, and $e=n$ if $\mathcal D$ is a cycle. If $\mathcal D$ is a cycle then define $\theta_n = \theta_0$.
 For the graph $\mathcal D$ the vertices $\theta_{i-1}, \theta_i$ are adjacent for $1 \leq i \leq e$. By this and
 Lemma \ref{lem:3v},
 \begin{align*}
 \theta_{i-1} - (q^2+ q^{-2}) \theta_i + \theta_{i+1} = 0 \qquad \qquad (1 \leq i \leq e-1).
 \end{align*}
 For this recurrence the characteristic polynomial is 
 \begin{align*}
 1-(q^2+q^{-2})x+x^2 = (x-q^2)(x-q^{-2}).
 \end{align*}
 The roots $q^2$, $q^{-2}$ are distinct since $q$ is not a root of 1. By these comments the recurrence has 
 general solution 
 \begin{align}
 \theta_i = a q^{d-2i} + \alpha q^{2i-d} \qquad \qquad (0 \leq i \leq e),
 \label{eq:eee}
 \end{align}
 where $a, \alpha \in \mathbb F$.
 Using this general solution,
 \begin{align*}
 0 = P(\theta_0, \theta_1) = (q^2-q^{-2})^2 (1-a \alpha).
 \end{align*}
 The scalar $q^2-q^{-2}$ is nonzero since 
 $q$ is not a root of unity, so $a\alpha=1$. Therefore $a\not=0$ and $\alpha = a^{-1}$. It remains to show that the graph $\mathcal D$ is not a cycle.
 Assume that $\mathcal D$
 is a cycle.
 By construction $\theta_n = \theta_0$. Also by construction $n\geq 3$, so $d=n-1\geq 2$. Using (\ref{eq:eee})  with $\alpha = a^{-1}$,
 \begin{align*}
 0 = \frac{\theta_n-\theta_0}{\theta_d-\theta_1}= \frac{q^{2n} -1}{q^{2d} - q^2}.
 \end{align*}
 Therefore
 $q^{2n}=1$, contradicting our assumption
 that $q$ is not a root of unity. Consequently $\mathcal D$ is not a cycle, and the result follows.
 \end{proof}
 
 \noindent For notational convenience, abbreviate $E_i = E_{\theta_i}$ for $0 \leq i \leq d$. 
 By the construction and (\ref{eq:choices}),
 \begin{align} 
 \label{eq:triple}
           E_i A^* E_j = 0 \quad {\mbox{\rm if}}  \quad   \vert  i - j \vert > 1 \qquad (0 \leq i,j\leq d).
          \end{align}
     \noindent For $0 \leq i \leq d$ let $V_i$ denote the $\theta_i$-eigenspace for $A$ on $V$. So
 $V_i = E_iV$. 
 
 \begin{lemma}  
 \label{lem:AsVVV}
 We have
 \begin{align}
 A^* V_i \subseteq V_{i-1} + V_i + V_{i+1} \qquad \qquad (0 \leq i \leq d),
 \label{eq:threeterm}
 \end{align}
 where $V_{-1} = 0$ and $V_{d+1} = 0$.
 \end{lemma}
 \begin{proof} We have
 \begin{align*}
 A^* V_i = I A^* E_iV = \sum_{\ell=0}^d E_\ell A^* E_iV.
 \end{align*}
 For $0 \leq \ell \leq d$, $E_\ell A^* E_i = 0 $ if $\vert \ell - i \vert > 1$.
 Also for $0 \leq \ell \leq d$, $E_\ell A^* E_iV \subseteq E_\ell V = V_\ell$.
 The result follows.
 \end{proof}
 \noindent Lemma  \ref{lem:AsVVV}  shows how $A^*$ acts on the eigenspaces  for $A$ on $V$. Interchanging the roles of $A$, $A^*$ we see that $A$ acts on the
 eigenspaces for $A^*$ on $V$ in the following way.
 \begin{lemma}  
 \label{lem:AVsVsVs}
 There exists an ordering $\lbrace V^*_i \rbrace_{i=0}^\delta$ of the eigenspaces for $A^*$ on $V$ such that
 \begin{align}
 A V^*_i \subseteq V^*_{i-1} + V^*_i + V^*_{i+1} \qquad \qquad (0 \leq i \leq \delta),
 \end{align}
 where $V^*_{-1} = 0$ and $V^*_{\delta+1}=0$.
 \end{lemma}
 
 \noindent The concept of a tridiagonal pair was introduced in \cite[Definition~1.1]{TD00}. By Lemmas
  \ref{lem:AsVVV}, \ref{lem:AVsVsVs} the elements $A, A^*$ act on the $\mathcal O_q$-module $V$
  as a tridiagonal pair.
  \medskip
  
  \noindent 
 By \cite[Lemma~4.5]{TD00} the integers $d, \delta$ from Lemmas \ref{lem:AsVVV}, 
  \ref{lem:AVsVsVs} are equal; we call this common
  value the {\it diameter} of the $\mathcal O_q$-module $V$. For $0 \leq i \leq d$ let $\theta^*_i$ denote the
  eigenvalue 
  associated with the eigenspace $V^*_i$ for $A^*$ on $V$.
  \begin{lemma} 
  \label{eq:ths} 
  With the above notation, there exists
  $0 \not=b \in \mathbb F$ such that
 \begin{align*}
 \theta^*_i = b q^{d-2i} + b^{-1} q^{2i-d} \qquad \qquad (0 \leq i \leq d).
 \end{align*}
 \end{lemma}
 \begin{proof} Interchange the roles of $A$ and $A^*$ in Lemma
 \ref{lem:form}.
 \end{proof}
 \noindent We have a comment.
 \begin{lemma}  Neither of $a^2$, $b^2$ is among $q^{2d-2}, q^{2d-4}, \ldots, q^{2-2d}$.
 \end{lemma}
 \begin{proof} Use Lemmas \ref{lem:form},  \ref{eq:ths} 
 and the fact that 
  the scalars $\lbrace \theta_i\rbrace_{i=0}^d$ are mutually distinct and the
 scalars $\lbrace \theta^*_i\rbrace_{i=0}^d$ are mutually distinct.
 \end{proof}

 \noindent We mentioned earlier that the elements $A$, $A^*$ act on the $\mathcal O_q$-module $V$
 as a tridiagonal pair. By the form of the eigenvalues
 $\lbrace \theta_i \rbrace_{i=0}^d$
 (resp. $\lbrace \theta^*_i \rbrace_{i=0}^d$) in Lemma
 \ref{lem:form} (resp. Lemma \ref{eq:ths}), we see that this tridiagonal pair has $q$-Racah type in the sense of 
\cite[Section~1]{terLop}.

 \section{The Lusztig automorphism of $\mathcal O_q$}
We continue to discuss the $\mathcal O_q$-module $V$ from Section 2. Recall the Lusztig automorphism $L$ of $\mathcal O_q$
from Lemma \ref{lem:Lusztig}.
In \cite[Section~8]{terLusztig} we briefly described the action of $L(A^*)$ on $V$. In the present 
section we describe this action in greater detail.
Referring to the equation on the right in
(\ref{eq:Lusztig}), for $0 \leq i,j\leq d$ we multiply each side on the left by $E_i$ and the right by $E_j$. This yields
\begin{align}
E_i L(A^*) E_j = E_i A^* E_j t_{ij},
\label{eq:ELE}
\end{align}
where
\begin{align}
\label{eq:tij}
t_{ij} & = 1 + \frac{\theta_i - \theta_j}{q-q^{-1}}\,\frac{ q \theta_i - q^{-1} \theta_j}{q^2-q^{-2}}.
\end{align}
If $\vert i-j \vert > 1$ then
 each side of  (\ref{eq:ELE})
 is equal to zero, in view of
 (\ref{eq:triple}).
We now consider
 (\ref{eq:ELE}) for $\vert i-j\vert \leq 1$. Note that
\begin{align}
\label{eq:tii}
 t_{ii} = 1  \qquad \qquad  (0 \leq i \leq d).
 \end{align}
 
\begin{lemma} 
\label{lem:tijhelp}
For $0 \leq i,j\leq d$ such that $\vert i-j\vert = 1$,
\begin{align*}
\frac{q \theta_i - q^{-1} \theta_j}{q^2-q^{-2}} \,\frac{q \theta_j - q^{-1}\theta_i}{q^2-q^{-2}} = 1.
\end{align*}
\end{lemma}

\begin{proof} This is a reformulation of the equation $P(\theta_i, \theta_j) = 0$.
\end{proof}

\begin{lemma} \label{lem:tt}
For $0 \leq i,j\leq d$ such that $\vert i-j\vert = 1$,
\begin{align}
t_{ij} t_{ji} = 1.
\label{eq:tt}
\end{align}
\end{lemma}
\begin{proof} To verify (\ref{eq:tt}), evaluate the left-hand side using
$(\ref{eq:tij})$ and simplify the result using Lemma
\ref{lem:tijhelp}.
\end{proof}

\begin{definition} 
\label{def:ti} 
\rm Define $
t_i = t_{01} t_{12} \cdots t_{i-1,i} $ for $0 \leq i \leq d$. We interpret $t_0 = 1$.
\end{definition}
\begin{lemma} \label{lem:nonzerot} We have
$t_i \not=0$ for $0 \leq i \leq d$.
\end{lemma}
\begin{proof} By Definition
\ref{def:ti} and since each of $t_{01}, t_{12}, \ldots, t_{i-1,i}$ is nonzero by
Lemma
\ref{lem:tt}.
\end{proof}
\begin{lemma} 
\label{lem:tivstj} 
For $0\leq i,j\leq d$ such that  $\vert i-j\vert \leq 1$,
\begin{align*}
t_{ij} = \frac{t_j}{t_i}.
\end{align*}
\end{lemma}
\begin{proof} For $i=j$ use (\ref{eq:tii}). For $\vert i-j\vert = 1$ use Lemma
\ref{lem:tt} and Definition \ref{def:ti}.
\end{proof}

\noindent Define $H \in {\rm End}(V)$ by 
\begin{align}
\label{eq:psi}
H = \sum_{i=0}^d t_i E_i,
\end{align}
\noindent where $\lbrace t_i \rbrace_{i=0}^d$ are from Definition
\ref{def:ti}. The map $H$ is invertible and
\begin{align}
\label{eq:psiInv}
H^{-1}  = \sum_{i=0}^d t^{-1}_i E_i.
\end{align}

\begin{lemma} \label{lem:PP} The following hold on $V$:
\begin{enumerate}
\item[\rm (i)] $L(A) = H^{-1} A H$;
\item[\rm (ii)] $L(A^*) = H^{-1} A^*H$.
\end{enumerate}
\end{lemma} 
\begin{proof} (i) We show that each side is equal to $A$.
By 
(\ref{eq:Lusztig}) we have $L(A)=A$. For $0 \leq i \leq d$ we have $AE_i = \theta_i E_i = E_i A$, so $E_i$ commutes with $A$.
By this and (\ref{eq:psi}) we see that $H$ commutes with $A$. Therefore $H^{-1}AH=A$.
\\
\noindent (ii)  Define $\Delta = L(A^*) - H^{-1} A^* H$. We show that $\Delta=0$. Using $I = E_0 +\cdots + E_d$ we obtain
\begin{align*}
\Delta = I \Delta I = \sum_{i=0}^d \sum_{j=0}^d E_i \Delta E_j.
\end{align*}
For $0 \leq i,j\leq d$ we show that $E_i \Delta E_j = 0$.
Using
(\ref{eq:ELE}) and
(\ref{eq:psi}),
(\ref{eq:psiInv})
we obtain
\begin{align*}
E_i \Delta E_j = E_i A^* E_j (t_{ij} - t^{-1}_i t_j).
\end{align*}
If $\vert i-j\vert > 1$ then $E_i A^* E_j = 0$
by  (\ref{eq:triple}).
 If $\vert i-j \vert \leq 1$ then $t_{ij} = t^{-1}_i t_j$ by Lemma
 \ref{lem:tivstj}.
Therefore $E_i \Delta E_j = 0$. We have shown that $E_i \Delta E_j=0$ for $0 \leq i,j\leq d$, so $\Delta = 0$. The result follows.
\end{proof}

\begin{proposition} \label{prop:L} For  $X \in \mathcal O_q$ the following hold on $V$:
\begin{align}
 L(X) = H^{-1} X H, \qquad \qquad
 L^{-1} (X) = HX H^{-1}.
 \label{eq:LPXP}
 \end{align}
 \end{proposition} 
 \begin{proof}  By Lemma \ref{lem:PP}
 and since the algebra $\mathcal O_q$ is generated by $A$, $A^*$. 
\end{proof}
 
 \noindent We will discuss the implications of Proposition \ref{prop:L}, after reviewing a few concepts.
 \begin{definition}
 \label{def:mi}
 \rm Let $\mathcal A$ denote an algebra and let $W$, $W'$ denote 
 $\mathcal A$-modules. By an {\it isomorphism of $\mathcal A$-modules from $W$ to $W'$} we mean
 an $\mathbb F$-linear bijection $\gamma : W \to W'$ such that $\xi \gamma = \gamma \xi$ on $W$
 for all $\xi \in \mathcal A$. The $\mathcal A$-modules $W$, $W'$ are said to be {\it isomorphic} whenever
 there exists an isomorphism of $\mathcal A$-modules from $W$ to $W'$.
 \end{definition}

 \begin{definition}
 \label{def:twist}
 \rm Let $\sigma$ denote an automorphism of an algebra $\mathcal A$. Let $W$ denote an $\mathcal A$-module. There exists
 an $\mathcal A$-module structure on $W$, called {\it $W$ twisted via $\sigma$}, that behaves as follows: for all $\xi \in \mathcal A$ and $w \in W$,
the vector $\xi .w$ computed in $W$ twisted via $\sigma$ coincides with the vector $\sigma^{-1} (\xi).w$ computed in the original $\mathcal A$-module $W$.
Sometimes we abbreviate ${}^\sigma W$ for $W$ twisted via $\sigma$.
\end{definition}
 
 \begin{proposition}  
 \label{cor:iso}
 The following $\mathcal O_q$-modules are isomorphic:
 \begin{enumerate}
 \item[\rm (i)] the $\mathcal O_q$-module $V$;
 \item[\rm (ii)] the $\mathcal O_q$-module $V$ twisted via $L$;
  \item[\rm (iii)] the $\mathcal O_q$-module $V$ twisted via $L^{-1}$.
 \end{enumerate}
 Moreover, the map $H$ from {\rm (\ref{eq:psi})}
is an isomorphism of $\mathcal O_q$-modules from {\rm (i)} to {\rm (ii)} and from {\rm (iii)} to {\rm (i)}.
 \end{proposition}
 \begin{proof} It suffices to prove the last assertion in the proposition statement. By the construction and Proposition \ref{prop:L}, the map $H : V \to V$
is an $\mathbb F$-linear bijection such that 
$L^{-1}(X)H = HX$ and
$H L(X) = X H$
on $V$
for all $X \in \mathcal O_q$. By this and Definitions
 \ref{def:mi},  \ref{def:twist} we get the last assertion in the proposition statement.
 \end{proof}
 
 \noindent
 We mentioned below
 Lemma \ref{lem:AVsVsVs}
 that $A$, $A^*$ act on the $\mathcal O_q$-module $V$ as a tridiagonal pair. Next we express Proposition \ref{cor:iso} in terms  of tridiagonal pairs. The notion of isomorphism
 for tridiagonal pairs is given in \cite[Definition~3.1]{NomTerSharp}. 
  
 \begin{proposition}
 \label{thm:isoTD}
 The following are isomorphic tridiagonal pairs:
 \begin{enumerate}
 \item[\rm (i)] the action of $A$, $A^*$ on $V$;
 \item[\rm (ii)] the action of $A$, $L(A^*)$ on $V$;
  \item[\rm (iii)] the action of $A$, $L^{-1}(A^*)$ on $V$.
 \end{enumerate}
 Moreover, the map
 $H$ from {\rm (\ref{eq:psi})}
is an isomorphism of tridiagonal pairs from {\rm (ii)} to {\rm (i)} and from {\rm (i)} to {\rm (iii)}.
 \end{proposition}
 
 \noindent We emphasize a few points about $L(A^*)$ and $L^{-1}(A^*)$.
 \begin{lemma} 
 \label{lem:LAs} For $\varepsilon \in \lbrace 1,-1\rbrace $
 the element $L ^\varepsilon (A^*)$ is diagonalizable on $V$, with eigenvalues
 $\lbrace \theta^*_i\rbrace_{i=0}^d$ and $\theta^*_i$-eigenspace $H^{-\varepsilon}V^*_i$ 
  for $0 \leq i \leq d$.
 \end{lemma}
 \begin{proof} By construction the element $A^*$ is diagonalizable on $V$, with eigenvalues
 $\lbrace \theta^*_i\rbrace_{i=0}^d$ and $\theta^*_i$-eigenspace $V^*_i$ for $0 \leq i \leq d$.
The result follows from this and Proposition
\ref{prop:L} .
 \end{proof}
 
 \begin{definition}
 \label{def:shortnot}
\rm For notational convenience,  we abbreviate $V^+_i = H^{-1} V^*_i$ and $V^-_i = HV^*_i$ for $0 \leq i \leq d$.
\end{definition}

 \noindent Our next goal for this section is to develop some formulas concerning $H$ that will be used in
 later sections.
 
 \begin{lemma}
 \label{lem:tijaq}
 We have
 \begin{align}
 t_{i-1,i} = a^2 q^{2(d-2i+1)} \qquad \qquad (1 \leq i \leq d).
 \end{align}
 \end{lemma}
 \begin{proof} Use  Lemma  \ref{lem:form} and (\ref{eq:tij}).
 \end{proof}
 
 \begin{lemma} 
 \label{lem:tiform}
 We have
 \begin{align}
 t_i = a^{2i} q^{2i(d-i)} \qquad \qquad (0 \leq i \leq d).
 \label{eq:tiform}
 \end{align}
 \end{lemma}
 \begin{proof} Use Definition
 \ref{def:ti} and
 Lemma \ref{lem:tijaq}.
 \end{proof}
 \noindent We recall some notation. For $z,t \in \mathbb F$,
 \begin{align*}
 (z,t)_n = (1-z)(1-zt) \cdots (1-zt^{n-1})
 \qquad \qquad (n \in \mathbb N).
 \end{align*}
 We will be discussing basic hypergeometric series, using the notation of
 \cite{gr,koekoek}.

 \begin{lemma}
 \label{lem:sum}
 For $0 \leq r\leq s \leq d$,
 \begin{align}
 \label{eq:chu}
 \frac{t_s}{t_r} &= \sum_{i=0}^{s-r} \frac{ a^i q^{i(d-2r)} (\theta_s - \theta_r)(\theta_s - \theta_{r+1} )\cdots (\theta_s - \theta_{r+i-1})}{(q^2;q^2)_i},
 \\
 \frac{t_r}{t_s} &= \sum_{i=0}^{s-r} \frac{ a^{-i} q^{i(2r-d)} (\theta_s - \theta_r)(\theta_s - \theta_{r+1} )\cdots (\theta_s - \theta_{r+i-1})}{(q^{-2};q^{-2})_i}.
\label{eq:chuInv}
 \end{align}
 \end{lemma}
 \begin{proof} To verify
 (\ref{eq:chu}), 
 evaluate the left-hand side using Lemma \ref{lem:tiform}
 and the right-hand side using Lemma
  \ref{lem:form}. The result becomes  a special case of the basic
  Chu/Vandermonde summation formula \cite[p.~354]{gr}:
  \begin{align*}
  a^{2s-2r}q^{2(s-r)(d-r-s)} = 
 {}_2\phi_1 \biggl(
 \genfrac{}{}{0pt}{}
 {q^{2s-2r}, a^2 q^{2d-2r-2s}}
  {0 }
  \,\bigg\vert \, q^{-2};  q^{-2} \biggr).
\end{align*}
We have verified 
(\ref{eq:chu}). To obtain (\ref{eq:chuInv}) from (\ref{eq:chu}), replace $q\mapsto q^{-1}$ and $a\mapsto a^{-1}$.
 \end{proof}
 
 \noindent Recall the eigenspaces $\lbrace V_i \rbrace_{i=0}^d $ for $A$ on $V$.
 \begin{proposition} 
 \label{prop:CVP}
 For $0 \leq r \leq d$ the following holds on $V_r + V_{r+1} + \cdots + V_d$:
 \begin{align}
 \label{eq:PsiExpand}
 H &= t_r \sum_{i=0}^{d-r} \frac{ a^i q^{i(d-2r)}(A-\theta_r I )( A-\theta_{r+1}I ) \cdots (A-\theta_{r+i-1}I)}{(q^2;q^2)_i},
 \\
  \label{eq:PsiExpandInv}
 H^{-1} &= t^{-1}_r \sum_{i=0}^{d-r} \frac{ a^{-i} q^{i(2r-d)}(A-\theta_r I )( A-\theta_{r+1}I ) \cdots (A-\theta_{r+i-1}I)}{(q^{-2};q^{-2})_i}.
 \end{align}
 \end{proposition}
 \begin{proof} To verify (\ref{eq:PsiExpand}), use
 (\ref{eq:psi}), (\ref{eq:chu})
 to see that 
 for $r \leq s \leq d$ the $V_s$-eigenvalue for either side of
 (\ref{eq:PsiExpand}) is equal to $t_s$. We have verified 
 (\ref{eq:PsiExpand}).
 To verify (\ref{eq:PsiExpandInv}), use
 (\ref{eq:psiInv}), (\ref{eq:chuInv})
 to see that 
 for $r \leq s \leq d$ the $V_s$-eigenvalue for either side of
 (\ref{eq:PsiExpandInv}) is equal to $t^{-1}_s$. We have verified 
 (\ref{eq:PsiExpandInv}).
 \end{proof}
 
 \begin{proposition} 
 \label{cor:CVP}
 The following holds on $V$:
 \begin{align*}
 H&=  \sum_{i=0}^{d} \frac{ a^i q^{id}(A-\theta_0 I )( A-\theta_{1}I ) \cdots (A-\theta_{i-1}I)}{(q^2;q^2)_i},
 \\
 H^{-1} &= \sum_{i=0}^{d} \frac{ a^{-i} q^{-id}(A-\theta_0 I )( A-\theta_{1}I ) \cdots (A-\theta_{i-1}I)}{(q^{-2};q^{-2})_i}.
 \end{align*}
 \end{proposition}
 \begin{proof} Set $r=0$ in Proposition  \ref{prop:CVP}.
 \end{proof}

 \noindent We mention a variation on Lemma 
 \ref{lem:sum}
 and 
 Propositions  \ref{prop:CVP}, \ref{cor:CVP}.

 \begin{lemma}
 \label{lem:sumVar}
 For $0 \leq r\leq s \leq d$,
 \begin{align}
 \label{eq:chuVar}
 \frac{t_r}{t_s} &= \sum_{i=0}^{s-r} \frac{ a^{-i} q^{i(2s-d)} (\theta_r - \theta_s)(\theta_r - \theta_{s-1} )\cdots (\theta_r - \theta_{s-i+1})}{(q^2;q^2)_i},
 \\
 \frac{t_s}{t_r} &= \sum_{i=0}^{s-r} \frac{ a^{i} q^{i(d-2s)} (\theta_r - \theta_s)(\theta_r - \theta_{s-1} )\cdots (\theta_r - \theta_{s-i+1})}{(q^{-2};q^{-2})_i}.
\label{eq:chuInvVar}
 \end{align}
 \end{lemma}
 \begin{proof} 
 To verify
 (\ref{eq:chuVar}), 
 evaluate the left-hand side using Lemma \ref{lem:tiform}
 and the right-hand side using Lemma
  \ref{lem:form}. The result becomes  a special case of the basic
  Chu/Vandermonde summation formula \cite[p.~354]{gr}:

  \begin{align*}
  a^{2r-2s}q^{2(r-s)(d-r-s)} = 
 {}_2\phi_1 \biggl(
 \genfrac{}{}{0pt}{}
 {q^{2s-2r}, a^{-2} q^{2r+2s-2d}}
  {0 }
  \,\bigg\vert \, q^{-2};  q^{-2} \biggr).
\end{align*}
We have verified 
(\ref{eq:chuVar}). To obtain (\ref{eq:chuInvVar}) from (\ref{eq:chuVar}), replace $q\mapsto q^{-1}$ and $a\mapsto a^{-1}$.
 \end{proof}

 \begin{proposition} 
 \label{prop:CVPVar}
 For $0 \leq s \leq d$ the following holds on $V_0 + V_{1} + \cdots + V_s$:
 \begin{align*}
 H &= t_s \sum_{i=0}^{s} \frac{ a^{-i} q^{i(2s-d)}(A-\theta_s I )( A-\theta_{s-1}I ) \cdots (A-\theta_{s-i+1}I)}{(q^2;q^2)_i},
 \\
 H^{-1} &= t^{-1}_s \sum_{i=0}^{s} \frac{ a^{i} q^{i(d-2s)}(A-\theta_s I )( A-\theta_{s-1}I ) \cdots (A-\theta_{s-i+1}I)}{(q^{-2};q^{-2})_i}.
 \end{align*}
 \end{proposition}
 \begin{proof} Similar to the proof of
 Proposition \ref{prop:CVP}.
 \end{proof}

 \begin{proposition} 
 \label{cor:CVPVar}
 The following holds on $V$:
 \begin{align*}
 H &= t_d \sum_{i=0}^{d} \frac{ a^{-i} q^{id}(A-\theta_d I )( A-\theta_{d-1}I ) \cdots (A-\theta_{d-i+1}I)}{(q^2;q^2)_i},
 \\
 H^{-1} &= t^{-1}_d \sum_{i=0}^{d} \frac{ a^{i} q^{-id}(A-\theta_d I )( A-\theta_{d-1}I ) \cdots (A-\theta_{d-i+1}I)}{(q^{-2};q^{-2})_i}.
 \end{align*}
 \end{proposition}
 \begin{proof} Set $s=d$ in 
 Proposition \ref{prop:CVPVar}.
 \end{proof}


\section{Some diagrams}
\noindent We continue to discuss the $\mathcal O_q$-module $V$ from Section 2. 
In this section we introduce some diagrams that will help us describe
how  the actions of $A$, $A^*$, $L^{\pm 1}(A^*)$ on $V$ are related.

\begin{definition}\rm By a {\it decomposition of $V$} we mean a sequence $\lbrace W_i\rbrace_{i=0}^d$ of nonzero subspaces whose direct sum
is $V$. 
\end{definition}

\begin{example} \rm The eigenspaces $\lbrace V_i\rbrace_{i=0}^d$ 
of $A$ (resp. $\lbrace V^*_i \rbrace_{i=0}^d$  of $A^*$)
(resp. $\lbrace V^+_i\rbrace_{i=0}^d$ of $L(A^*)$)
(resp. $\lbrace V^-_i\rbrace_{i=0}^d$ of $L^{-1}(A^*)$)
form a decomposition of $V$.
\end{example}

\begin{definition} \rm
Let  $\lbrace W_i\rbrace_{i=0}^d$  denote a decomposition of $V$.  Its {\it inversion} is the decomposition $\lbrace W_{d-i}\rbrace_{i=0}^d$ of $V$.
\end{definition}

\noindent  Let $\lbrace W_i \rbrace_{i=0}^d$ denote a decomposition of $V$.
We will describe this decomposition by the diagram
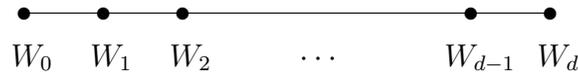
\begin{center}
\begin{picture}(0,20)
\put(-100,0){\line(1,0){200}}
\put(-102,-3){$\bullet$}	
\put(-104,-20){$W_0$}
\put(-72,-3){$\bullet$}	
\put(-74,-20){$W_1$}
\put(-42,-3){$\bullet$}
\put(-44,-20){$W_2$}
\put(5,-20) {$\cdots $} 
	\put(67,-3){$\bullet$}	
	\put(60,-20){$W_{d-1}$}
	
\put(97,-3){$\bullet$}	
\put(95,-20){$W_d$}
\end{picture}
\end{center}
\vspace{1cm}

\noindent  The labels $W_i$ might be suppressed, if they are clear from the context.
\medskip

\noindent Let $\lbrace W_i\rbrace_{i=0}^d$ and 
 $\lbrace W'_i\rbrace_{i=0}^d$ denote decompositions of $V$.
The condition
\begin{align*} 
W_0 + W_1 + \cdots + W_i = W'_0 + W'_1 + \cdots + W'_i \qquad \qquad (0 \leq i \leq d)
\end{align*}
will be described by the diagram
\medskip

\begin{center}
\begin{picture}(0,60)
\put(-100,0){\line(5,1){200}}
\put(-100,0){\line(5,-1){200}}

\put(-102,-3){$\bullet$}	
\put(-72,-9){$\bullet$}	
\put(-42,-15){$\bullet$}	
\put(67,-37){$\bullet$}	
\put(97,-43){$\bullet$}

\put(-109,-20){$W'_0$}
\put(-79,-26){$W'_1$}
\put(-49,-32){$W'_2$}
\put(-5,-38) {$\cdot$}	
\put(5,-40) {$\cdot$}	
\put(15,-42) {$\cdot$}	
\put(55,-54){$W'_{d-1}$}
\put(90,-60){$W'_d$}

\put(-102,-3){$\bullet$}	
\put(-72,3){$\bullet$}	
\put(-42,9){$\bullet$}	
\put(67,31){$\bullet$}	
\put(97,37){$\bullet$}

\put(-109,12){$W_0$}
\put(-79,18){$W_1$}
\put(-49,24){$W_2$}
\put(-5,31) {$\cdot$}	
\put(5,33) {$\cdot$}	
\put(15,35) {$\cdot$}	
\put(55,46){$W_{d-1}$}
\put(90,52){$W_d$}

\end{picture}
\end{center}
\vspace{2cm}

\newpage
\begin{lemma} 
\label{lem:split} {\rm (See \cite[Theorem~4.6]{TD00}.)}
There exist decompositions of $V$ that are related to $\lbrace V_i\rbrace_{i=0}^d$ and $\lbrace V^*_i\rbrace_{i=0}^d$ in the following way:

\begin{center}
\begin{picture}(300,230)
\put(0,180){\line(2,0){300}}
\put(148,213){$A^*$}
\put(-3,177){$\bullet$}
\put(27,177){$\bullet$}
\put(57,177){$\bullet$}
\put(237,177){$\bullet$}
\put(267,177){$\bullet$}
\put(297,177){$\bullet$}

\put(-4,193){$V^*_0$}
\put(26,193){$V^*_1$}
\put(56,193){$V^*_2$}
\put(230,193){$V^*_{d-2}$}
\put(260,193){$V^*_{d-1}$}
\put(294,193){$V^*_d$}
\put(150,193){$\cdot$}
\put(140,193){$\cdot$}
\put(160,193){$\cdot$}

\put(-3,157){$\bullet$}
\put(-3,137){$\bullet$}

\put(27,159){$\bullet$}
\put(57,141){$\bullet$}

\put(297,157){$\bullet$}
\put(297,137){$\bullet$}

\put(267,159){$\bullet$}
\put(237,141){$\bullet$}

\put(0,180){\line(0,-1){90}}
\put(0,90){\line(0,-1){90}}

\put(0,180){\line(30,-18){100}}
\put(100,120){\line(30,-18){200}}

\put(300,180){\line(-30,-18){100}}
\put(200,120){\line(-30,-18){45}}
\put(0,0){\line(30,18){145}}

\put(300,180){\line(0,-1){90}}
\put(300,90){\line(0,-1){90}}

\put(0,0){\line(1,0){300}}
\put(147,-40){$A$}

\put(-3,-3){$\bullet$}
\put(27,-3){$\bullet$}
\put(57,-3){$\bullet$}
\put(237,-3){$\bullet$}
\put(267,-3){$\bullet$}
\put(297,-3){$\bullet$}

\put(-3,17){$\bullet$}
\put(-3,37){$\bullet$}
\put(27,15){$\bullet$}
\put(57,33){$\bullet$}

\put(297,37){$\bullet$}
\put(297,17){$\bullet$}

\put(237,33){$\bullet$}
\put(267,15){$\bullet$}

\put(-4,-20){$V_0$}
\put(26,-20){$V_1$}
\put(56,-20){$V_2$}
\put(230,-20){$V_{d-2}$}
\put(260,-20){$V_{d-1}$}
\put(294,-20){$V_d$}
\put(150,-20){$\cdot$}
\put(140,-20){$\cdot$}
\put(160,-20){$\cdot$}

\end{picture}
\end{center}
\end{lemma}
\vspace{1.5cm}

\begin{lemma} 
\label{lem:Lsplit} 
There exist decompositions of $V$ that are related to $\lbrace V_i\rbrace_{i=0}^d$ and $\lbrace V^+_i\rbrace_{i=0}^d$ in the following way:

\begin{center}
\begin{picture}(300,230)
\put(0,180){\line(2,0){300}}
\put(138,213){$L(A^*)$}
\put(-3,177){$\bullet$}
\put(27,177){$\bullet$}
\put(57,177){$\bullet$}
\put(237,177){$\bullet$}
\put(267,177){$\bullet$}
\put(297,177){$\bullet$}

\put(-6,193){$V^+_0$}
\put(24,193){$V^+_1$}
\put(54,193){$V^+_2$}
\put(230,193){$V^+_{d-2}$}
\put(260,193){$V^+_{d-1}$}
\put(294,193){$V^+_d$}
\put(150,193){$\cdot$}
\put(140,193){$\cdot$}
\put(160,193){$\cdot$}

\put(-3,157){$\bullet$}
\put(-3,137){$\bullet$}

\put(27,159){$\bullet$}
\put(57,141){$\bullet$}

\put(297,157){$\bullet$}
\put(297,137){$\bullet$}

\put(267,159){$\bullet$}
\put(237,141){$\bullet$}

\put(0,180){\line(0,-1){90}}
\put(0,90){\line(0,-1){90}}

\put(0,180){\line(30,-18){100}}
\put(100,120){\line(30,-18){200}}

\put(300,180){\line(-30,-18){100}}
\put(200,120){\line(-30,-18){45}}
\put(0,0){\line(30,18){145}}

\put(300,180){\line(0,-1){90}}
\put(300,90){\line(0,-1){90}}

\put(0,0){\line(1,0){300}}
\put(147,-40){$A$}

\put(-3,-3){$\bullet$}
\put(27,-3){$\bullet$}
\put(57,-3){$\bullet$}
\put(237,-3){$\bullet$}
\put(267,-3){$\bullet$}
\put(297,-3){$\bullet$}

\put(-3,17){$\bullet$}
\put(-3,37){$\bullet$}
\put(27,15){$\bullet$}
\put(57,33){$\bullet$}

\put(297,37){$\bullet$}
\put(297,17){$\bullet$}

\put(237,33){$\bullet$}
\put(267,15){$\bullet$}

\put(-4,-20){$V_0$}
\put(26,-20){$V_1$}
\put(56,-20){$V_2$}
\put(230,-20){$V_{d-2}$}
\put(260,-20){$V_{d-1}$}
\put(294,-20){$V_d$}
\put(150,-20){$\cdot$}
\put(140,-20){$\cdot$}
\put(160,-20){$\cdot$}

\end{picture}
\end{center}
\vspace{1.5cm}
 We are using the notation in Definition \ref{def:shortnot}.
\end{lemma}
\begin{proof} By Proposition
  \ref{thm:isoTD}
and Lemma \ref{lem:split}.
\end{proof}

\begin{lemma} 
\label{lem:LsplitR}
There exist decompositions of $V$ that are related to $\lbrace V_i\rbrace_{i=0}^d$ and $\lbrace V^-_i\rbrace_{i=0}^d$ in the following way:

\begin{center}
\begin{picture}(300,230)
\put(0,180){\line(2,0){300}}
\put(132,213){$L^{-1}(A^*)$}
\put(-3,177){$\bullet$}
\put(27,177){$\bullet$}
\put(57,177){$\bullet$}
\put(237,177){$\bullet$}
\put(267,177){$\bullet$}
\put(297,177){$\bullet$}

\put(-6,193){$V^-_0$}
\put(24,193){$V^-_1$}
\put(54,193){$V^-_2$}
\put(230,193){$V^-_{d-2}$}
\put(260,193){$V^-_{d-1}$}
\put(294,193){$V^-_d$}
\put(150,193){$\cdot$}
\put(140,193){$\cdot$}
\put(160,193){$\cdot$}

\put(-3,157){$\bullet$}
\put(-3,137){$\bullet$}

\put(27,159){$\bullet$}
\put(57,141){$\bullet$}

\put(297,157){$\bullet$}
\put(297,137){$\bullet$}

\put(267,159){$\bullet$}
\put(237,141){$\bullet$}

\put(0,180){\line(0,-1){90}}
\put(0,90){\line(0,-1){90}}

\put(0,180){\line(30,-18){100}}
\put(100,120){\line(30,-18){200}}

\put(300,180){\line(-30,-18){100}}
\put(200,120){\line(-30,-18){45}}
\put(0,0){\line(30,18){145}}

\put(300,180){\line(0,-1){90}}
\put(300,90){\line(0,-1){90}}

\put(0,0){\line(1,0){300}}
\put(147,-40){$A$}

\put(-3,-3){$\bullet$}
\put(27,-3){$\bullet$}
\put(57,-3){$\bullet$}
\put(237,-3){$\bullet$}
\put(267,-3){$\bullet$}
\put(297,-3){$\bullet$}

\put(-3,17){$\bullet$}
\put(-3,37){$\bullet$}
\put(27,15){$\bullet$}
\put(57,33){$\bullet$}

\put(297,37){$\bullet$}
\put(297,17){$\bullet$}

\put(237,33){$\bullet$}
\put(267,15){$\bullet$}

\put(-4,-20){$V_0$}
\put(26,-20){$V_1$}
\put(56,-20){$V_2$}
\put(230,-20){$V_{d-2}$}
\put(260,-20){$V_{d-1}$}
\put(294,-20){$V_d$}
\put(150,-20){$\cdot$}
\put(140,-20){$\cdot$}
\put(160,-20){$\cdot$}

\end{picture}
\end{center}
\vspace{1.5cm}
 We are using the notation in Definition \ref{def:shortnot}.
\end{lemma}
\begin{proof} By Proposition
  \ref{thm:isoTD}
and Lemma \ref{lem:split}.
\end{proof}

\section{The maps $K$, $B$, $K^\downarrow$, $B^\downarrow$}
\noindent We continue to discuss the $\mathcal O_q$-module $V$ from Section 2.
 To aid in this discussion we bring in some maps
$K$, $B$, $K^\downarrow$, $B^\downarrow$.

\begin{definition}\rm 
Let $\lbrace W_i\rbrace_{i=0}^d$ denote a decomposition of $V$. The {\it corresponding map}
is the element $M \in {\rm End}(V)$ such that $(M-q^{d-2i}I)W_i = 0$ for $0 \leq i \leq d$.
In other words, $W_i$ is an eigenspace of $M$ with eigenvalue $q^{d-2i}$ for $0 \leq i \leq d$.
\end{definition}

\begin{example} \rm Let $M $ denote a diagonalizable element in ${\rm End}(V)$,
with eigenvalues $\lbrace q^{d-2i}\rbrace_{i=0}^d$. For $0 \leq i \leq d$
let $W_i$ denote the eigenspace of $M$ for the eigenvalue $q^{d-2i}$.
Then $\lbrace W_i \rbrace_{i=0}^d$ is a decomposition of $V$ and $M$
is the corresponding map.
\end{example}

\begin{lemma} For a decomposition $\lbrace W_i\rbrace_{i=0}^d$ of $V$ the following maps are inverses:
\begin{enumerate}
\item[\rm (i)] the map corresponding to  $\lbrace W_i\rbrace_{i=0}^d$;
\item[\rm (ii)] the map corresponding to  $\lbrace W_{d-i} \rbrace_{i=0}^d$.
\end{enumerate}
\end{lemma}

\newpage
\noindent  Let $\lbrace W_i \rbrace_{i=0}^d$ denote a decomposition of $V$, with
corresponding map $M$. We will describe $M$ using the diagram

\begin{center}
\begin{picture}(0,20)
\put(6,10){$M$}
\put(-100,0){\vector(1,0){115}}
\put(15,0){\line(1,0){85}}
\put(-102,-3){$\bullet$}	
\put(-104,-20){$W_0$}
\put(-72,-3){$\bullet$}	
\put(-74,-20){$W_1$}
\put(-42,-3){$\bullet$}
\put(-44,-20){$W_2$}
\put(5,-20) {$\cdots $} 
	\put(67,-3){$\bullet$}	
	\put(60,-20){$W_{d-1}$}
	
\put(97,-3){$\bullet$}	
\put(95,-20){$W_d$}
\end{picture}
\end{center}
\vspace{1cm}
We might suppress the labels $W_i$ along with the $\bullet$ notation, if they are clear from the context.
\medskip

\begin{definition}
\label{def:splitMap}
\rm
Referring to the diagram in Lemma
\ref{lem:split}, let $K$, $B$, $K^\downarrow$, $B^\downarrow$ denote the maps
 that correspond to the non-horizontal decompositions as shown below:

\begin{center}
\begin{picture}(300,220)
\put(0,180){\line(2,0){300}}
\put(143,193){$A^*$}

\put(0,180){\vector(0,-1){90}}
\put(0,90){\line(0,-1){90}}
\put(-18,90){$B$}

\put(0,180){\vector(30,-18){100}}
\put(100,120){\line(30,-18){200}}
\put(80,110){$K$}

\put(300,180){\vector(-30,-18){100}}
\put(200,120){\line(-30,-18){45}}
\put(0,0){\line(30,18){145}}
\put(210,110){$B^\downarrow$}

\put(300,180){\vector(0,-1){90}}
\put(300,90){\line(0,-1){90}}
\put(310,90){$K^\downarrow$}

\put(0,0){\line(1,0){300}}
\put(143,-20){$A$}
\end{picture}
\end{center}
\end{definition}

\vspace{1cm}
\noindent Below we cite some results about $K$ and $B$; similar results hold
for $K^\downarrow$ and $B^\downarrow$.
By \cite[Section~1.1]{augmented},
\begin{align}
&\frac{q K A - q^{-1} A K}{q-q^{-1}} = a K^2+ a^{-1} I,\qquad \qquad 
\frac{q B A - q^{-1} A B}{q-q^{-1}} =a^{-1} B^2+ a I.
\label{eq:KA}
\end{align}
\noindent By \cite[Theorem~9.9]{bockting},
\begin{align}
& aK^2 - \frac{a^{-1} q - a q^{-1}}{q-q^{-1}} K B - \frac{a q - a^{-1} q^{-1}}{q-q^{-1}} BK + a^{-1} B^2 = 0.
\label{eq:KB}
\end{align}
The  equations 
(\ref{eq:KA}),
(\ref{eq:KB})
can be reformulated as follows. By
\cite[Lemma~12.12]{bocktingTer},
\begin{align}
&\frac{q A K^{-1}  - q^{-1}K^{-1}A }{q-q^{-1}} = a^{-1} K^{-2}+ aI,\qquad \quad 
\frac{q A B^{-1} - q^{-1} B^{-1} A }{q-q^{-1}} =a B^{-2}+ a^{-1} I.
\label{eq:KAa}
\end{align}
\noindent By \cite[Theorem~9.10]{bockting},
\begin{align}
& a^{-1}K^{-2} - \frac{a^{-1} q - a q^{-1}}{q-q^{-1}} K^{-1} B^{-1} - \frac{a q - a^{-1} q^{-1}}{q-q^{-1}} B^{-1}K^{-1} + a B^{-2} = 0.
\label{eq:KBa}
\end{align}
\noindent We clarify why 
(\ref{eq:KB}), 
(\ref{eq:KBa}) are equivalent. They each assert that the maps
\begin{align*}
\frac{a^{-1}(q-q^{-1})}{a^{-1}-a} K^{-1}B - \frac{a^{-1} q- a q^{-1}}{a^{-1}-a}I,
\qquad \quad 
\frac{a^{-1}(q-q^{-1})}{a -a^{-1}} BK^{-1} - \frac{a q- a^{-1} q^{-1}}{a-a^{-1}}I
\end{align*}
are inverses, and also that the maps
\begin{align*}
\frac{a(q-q^{-1})}{a-a^{-1}} B^{-1}K - \frac{a q- a^{-1} q^{-1}}{a-a^{-1}}I,
\qquad \quad 
\frac{a (q-q^{-1})}{a^{-1} -a} KB^{-1} - \frac{a^{-1}q- a q^{-1}}{a^{-1}-a}I
\end{align*}
are inverses.

\section{How conjugation by $H$ affects $K$, $B$, $K^\downarrow$, $B^\downarrow$}
\noindent We continue to discuss the $\mathcal O_q$-module $V$ from Section 2.
In this section we describe what happens when one of
$K$,  $B$, $K^\downarrow$, $B^\downarrow$ is conjugated by $H$.

\begin{proposition}
\label{prop:four}
The maps $K$, $B$, $K^\downarrow$, $B^\downarrow$ from Definition
\ref{def:splitMap}
satisfy
\begin{align}
&H^{-1} B H = a A - a^{2} B^{-1},
\qquad \qquad \quad  \;\;
H^{-1} K H = a^{-1} A - a^{-2} K^{-1},
\label{eq:PKB}
\\
&H^{-1} B^\downarrow H = a A - a^{2} (B^\downarrow)^{-1},
 \qquad \qquad
H^{-1} K^\downarrow  H = a^{-1} A - a^{-2} (K^\downarrow)^{-1}.
\label{eq:PKBd}
\end{align}
\end{proposition}
\begin{proof}
We first obtain the equation on the right in
(\ref{eq:PKB}). Since $H$ commutes with $A$, it suffices to show that
\begin{align}
\label{eq:PsiCom}
H K^{-1} - K^{-1} H= a(A - aK - a^{-1}K^{-1} ) H.
\end{align}
\noindent For $0 \leq i \leq d$ let $U_i$ denote the eigenspace of $K$ for the eigenvalue $q^{d-2i}$. Thus $\lbrace U_i \rbrace_{i=0}^d$
is a decomposition of $V$, and $K$ is the corresponding map. For $0 \leq i \leq d$ the following holds on $U_i$:
\begin{align}
a K+ a^{-1} K^{-1} = \theta_i I .
\label{eq:aK}
\end{align}
By \cite[Theorem~4.6]{TD00},
\begin{align}
(A - \theta_i I ) U_i \subseteq U_{i+1} \qquad \quad (0 \leq i \leq d-1), \qquad (A-\theta_d I )U_d=0.
\label{eq:AU}
\end{align}
Define
\begin{align*}
R = A - aK - a^{-1} K^{-1}.
\end{align*}
By (\ref{eq:aK}), 
for 
$0 \leq i \leq d$ the following holds on $U_i$:
\begin{align}
\label{eq:RA}
R = A - \theta_i I.
\end{align}
By (\ref{eq:AU}) and (\ref{eq:RA}),
\begin{align}
\label{eq:RU}
R U_i \subseteq U_{i+1} \qquad \quad (0 \leq i \leq d-1), \qquad RU_d=0.
\end{align}
By (\ref{eq:RU}) and the construction,
\begin{align*}
RK = q^2 KR.
\end{align*}
For $0 \leq r \leq d$ we show that 
(\ref{eq:PsiCom}) holds on $U_r$.
Using
(\ref{eq:RA}), 
(\ref{eq:RU}) 
we find that for $0 \leq i \leq d-r$ the following holds on $U_r$:
\begin{align}
R^i = (A-\theta_r I ) (A-\theta_{r+1}I) \cdots (A-\theta_{r+i-1}I).
\label{eq:Ri}
\end{align}
Also by 
(\ref{eq:RU})  we have $R^{d-r+1}=0$ on $U_r$.
By Definition
\ref{def:splitMap}
and the discussion above Lemma \ref{lem:split},
\begin{align*} 
U_r + U_{r+1} + \cdots + U_d = V_r + V_{r+1} + \cdots + V_d.
\end{align*}
The above subspace contains $U_r$, so 
by Proposition
\ref{prop:CVP}
and (\ref{eq:Ri})
the following holds on $U_r$:
\begin{align*}
H = t_r \sum_{i=0}^{d-r} \frac{a^i R^i K^i}{(q^2;q^2)_i}.
\end{align*}
We may now argue that on $U_r$,
\begin{align*}
H K^{-1} - K^{-1} H &= t_r \sum_{i=0}^{d-r} \frac{a^i R^i K^i}{(q^2;q^2)_i} K^{-1}(1-q^{2i})
\\ 
&= t_r \sum_{i=1}^{d-r} \frac{a^i R^i K^{i-1}}{(q^2;q^2)_{i-1}}
\\
&= t_r \sum_{i=0}^{d-r-1} \frac{a^{i+1}R^{i+1} K^{i} }{(q^2;q^2)_{i}}
\\
&= t_r \sum_{i=0}^{d-r} \frac{a^{i+1}R^{i+1} K^{i}} {(q^2;q^2)_{i}}
\\
& = a R H
\\
& = a (A-aK-a^{-1} K^{-1}) H.
\end{align*}
We have obtained
(\ref{eq:PsiCom}), which implies  the equation on the right in
(\ref{eq:PKB}). The remaining equations in
(\ref{eq:PKB}), (\ref{eq:PKBd}) are obtained  by the following observations.
In (\ref{eq:PKB}), 
the equation on the left is obtained from the equation on the right
by replacing the decomposition $\lbrace V_i \rbrace_{i=0}^d$ by its inversion
$\lbrace V_{d-i}\rbrace_{i=0}^d$.
Moreover (\ref{eq:PKBd}) is obtained from (\ref{eq:PKB}) by replacing
the decomposition
$\lbrace V^*_i \rbrace_{i=0}^d$ by its inversion
$\lbrace V^*_{d-i}\rbrace_{i=0}^d$.

\end{proof}

\newpage
\begin{corollary}
\label{prop:LspiitMap}
Referring to the diagram in Lemma
\ref{lem:Lsplit},  the maps
 that correspond to the non-horizontal decompositions are shown below:

\begin{center}
\begin{picture}(300,220)
\put(0,180){\line(2,0){300}}
\put(140,193){$L(A^*)$}

\put(0,180){\vector(0,-1){90}}
\put(0,90){\line(0,-1){90}}
\put(-22,75){$aA-a^2 B^{-1}$}

\put(0,180){\vector(30,-18){100}}
\put(100,120){\line(30,-18){200}}
\put(60,102){$a^{-1}A-a^{-2}K^{-1}$}

\put(300,180){\vector(-30,-18){100}}
\put(200,120){\line(-30,-18){45}}
\put(0,0){\line(30,18){145}}
\put(170,102){$aA-a^2 (B^\downarrow)^{-1}$}

\put(300,180){\vector(0,-1){90}}
\put(300,90){\line(0,-1){90}}
\put(266,75){$a^{-1} A-a^{-2} (K^\downarrow)^{-1}$}

\put(0,0){\line(1,0){300}}
\put(143,-20){$A$}
\end{picture}
\end{center}
\vspace{1cm}
\end{corollary}
\begin{proof} By Propositions
\ref{prop:L},
\ref{prop:four}.
\end{proof}

\noindent The following result is a reformulation of Proposition
\ref{prop:four}.

\begin{proposition}
\label{prop:fourR}
The maps $K$, $B$, $K^\downarrow$, $B^\downarrow$ from Definition
\ref{def:splitMap}
satisfy
\begin{align}
&H B^{-1} H^{-1} = a^{-1} A - a^{-2} B,
\qquad \qquad \quad  \;\;
H K^{-1} H^{-1}= a A - a^{2} K,
\label{eq:PKBalt}
\\
&H (B^\downarrow)^{-1} H^{-1} = a^{-1} A - a^{-2} B^\downarrow,
 \qquad \qquad
H (K^\downarrow)^{-1} H^{-1}= a A - a^{2} K^\downarrow.
\label{eq:PKBdalt}
\end{align}
\end{proposition}
\begin{proof} These equations are reformulations of the equations in
(\ref{eq:PKB}),
(\ref{eq:PKBd}).
\end{proof}

\newpage
\begin{corollary}
\label{prop:LspiitMapR}
Referring to the diagram in Lemma
\ref{lem:LsplitR}, 
 the maps
 that correspond to the non-horizontal decompositions are shown below:

\begin{center}
\begin{picture}(300,220)
\put(0,180){\line(2,0){300}}
\put(130,193){$L^{-1}(A^*)$}

\put(0,0){\vector(0,1){95}}
\put(0,95){\line(0,1){85}}
\put(-33,75){$a^{-1}A-a^{-2} B$}

\put(0,180){\line(30,-18){90}}
\put(300,0){\vector(-30,18){210}}
\put(60,102){$a A-a^{2}K$}

\put(300,180){\line(-30,-18){90}}
\put(155,93){\vector(30,18){55}}
\put(0,0){\line(30,18){145}}
\put(180,102){$a^{-1}A-a^{-2} B^\downarrow$}

\put(300,0){\vector(0,1){95}}
\put(300,95){\line(0,1){85}}
\put(278,75){$a A-a^{2} K^\downarrow$}

\put(0,0){\line(1,0){300}}
\put(143,-20){$A$}
\end{picture}
\end{center}
\vspace{1cm}
\end{corollary}
\begin{proof} By Propositions
\ref{prop:L},
\ref{prop:fourR}.
\end{proof}

\noindent Shortly, we will make use of the following four maps:
\begin{align}
&\frac{a K - a^{-1} B}{a - a^{-1}}, \qquad \qquad
\frac{a^{-1} K^{-1} - a B^{-1}}{a^{-1}-a}.
\label{eq:fat}
\\
&\frac{a K^\downarrow- a^{-1} B^\downarrow}{a - a^{-1}}, \qquad \quad
\frac{a^{-1} (K^\downarrow)^{-1} - a (B^\downarrow)^{-1}}{a^{-1}-a}.
\label{eq:fat2}
\end{align}
The map on the left in (\ref{eq:fat}) first appeared in
\cite[Definition~6.1]{bocktingQexp}. Both maps in (\ref{eq:fat}) appear in
\cite[Section~12]{bocktingTer}.
By (\ref{eq:PKB}), 
\begin{align} 
H^{-1} 
\frac{a K - a^{-1} B}{a - a^{-1}} H &= 
\frac{a^{-1} K^{-1} - a B^{-1}}{a^{-1}-a}.
\label{eq:fatPsi}
\end{align}
By (\ref{eq:PKBd}),
\begin{align}
H^{-1} 
\frac{a K^\downarrow - a^{-1} B^\downarrow}{a - a^{-1}} H &= 
\frac{a^{-1} (K^\downarrow)^{-1} - a (B^\downarrow)^{-1}}{a^{-1}-a}.
\label{eq:fatPsi2}
\end{align}
\begin{lemma} 
\label{lem:fatInv}
Each map in 
{\rm (\ref{eq:fat})},
{\rm (\ref{eq:fat2})}
 is diagonalizable, with eigenvalues $\lbrace q^{d-2i}\rbrace_{i=0}^d$.
\end{lemma} 
\begin{proof} By
\cite[Lemma~8.1]{bocktingQexp}, the result holds for the map on the left in
(\ref{eq:fat}).
By this and (\ref{eq:fatPsi}), the result holds for the map on the right in
(\ref{eq:fat}). Replacing the decomposition $\lbrace V^*_i\rbrace_{i=0}^d$ by
its inversion $\lbrace V^*_{d-i}\rbrace_{i=0}^d$, we obtain the result for the maps
(\ref{eq:fat2}).
\end{proof}

\section{Equitable triples}

\noindent We continue to discuss the $\mathcal O_q$-module $V$ from Section 2.
As we compare the eigenspace decompositions for $A$, $A^*$, $L^{\pm 1}(A^*)$ on $V$, we
will use the concept of an equitable triple. We will define this concept after some preliminary remarks.
\medskip

\noindent
Let $X$, $Y$ denote elements in ${\rm End}(V)$ such that
\begin{align}
\frac{qXY-q^{-1}YX}{q-q^{-1}} = I.
\label{eq:qW}
\end{align}
We call (\ref{eq:qW}) the {\it $q$-Weyl relation}.

\begin{lemma} \label{lem:qW}
 For the above maps $X$, $Y$
the following {\rm (i), (ii)} hold:
\begin{enumerate}
\item[\rm (i)]
Let $u\in V$ denote an eigenvector for $X$ with nonzero eigenvalue $\lambda$. Then the vector
$(Y-\lambda^{-1} I)u$ is either zero, or an eigenvector for $X$ with eigenvalue $q^{-2} \lambda$.
\item[\rm (ii)]
Let $v \in V$ denote an eigenvector for $Y$ with nonzero eigenvalue $\mu$. Then the vector
$(X-\mu^{-1} I)v$ is either zero, or an eigenvector for $Y$ with eigenvalue $q^{2} \mu$.
\end{enumerate}
\end{lemma}
\begin{proof} (i) One checks that
\begin{align*}
(X - q^{-2}\lambda I )(Y -\lambda^{-1}I )u =0.
\end{align*}
\noindent (ii) Similar to the proof of (i) above.
\end{proof}

\begin{corollary} {\rm (See \cite[Lemma~11.4]{qtet}.)}
\label{cor:XY} For the above maps $X$, $Y$ the following are equivalent:
\begin{enumerate}
\item[\rm (i)] $X$ is diagonalizable with eigenvalues $\lbrace q^{d-2i} \rbrace_{i=0}^d$;
\item[\rm (ii)] $Y$ is diagonalizable with eigenvalues $\lbrace q^{d-2i} \rbrace_{i=0}^d$.
\end{enumerate}
Assume that  {\rm (i), (ii)} hold. Then the eigenspace decompositions of $X$, $Y$ are described 
by the following diagram:
\medskip

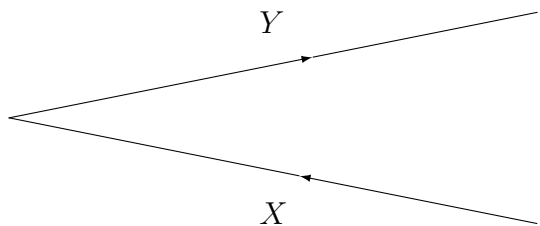
\begin{center}
\begin{picture}(0,40)
\put(-100,0){\vector(5,1){115}}
\put(15,23){\line(5,1){85}}
\put(-5,32){$Y$}
\put(100,-40){\vector(-5,1){90}}
\put(10,-22){\line(-5,1){110}}
\put(-5,-40){$X$}

\end{picture}
\end{center}
\vspace{2cm}
\end{corollary}
\begin{proof} This is a routine consequence of Lemma
\ref{lem:qW}.
\end{proof}

\begin{definition}
\label{def:et} \rm An {\it equitable triple on $V$} is 
a $3$-tuple $X,Y,Z$ of invertible elements in ${\rm End}(V)$ such that
\begin{align*}
\frac{qXY-q^{-1} YX}{q-q^{-1} } = I,
\qquad \quad 
\frac{qYZ-q^{-1} ZY}{q-q^{-1} } = I,
\qquad \quad 
\frac{qZX-q^{-1} XZ}{q-q^{-1} } = I.
\end{align*}
\end{definition}
\noindent  The references
\cite{equit}, \cite{fduq}
describe how equitable triples are related to the quantum group $U_q(\mathfrak{sl}_2)$.

\begin{proposition}
\label{thm:xyz}
For each row in the table below, the given 3-tuple $X$, $Y$, $Z$  is an equitable triple on $V$.
\bigskip

\centerline{
\begin{tabular}[t]{c|ccc}
   {\rm example} & $X$ & $Y$ &
   $Z$ 
\\
\hline
$1$  & 
$a A - a^2 K$ & $M^{-1}$ & $K$
\\
$2$ & $a^{-1} A - a^{-2} B$ & $ M^{-1}$ & $B$
\\
$3$ &$a A - a^2 K^\downarrow$ & $ (M^\downarrow)^{-1}$ & $ K^\downarrow$
\\
$4$&$a^{-1} A - a^{-2} B^\downarrow$ & $(M^\downarrow)^{-1}$ & $ B^\downarrow$
\\
$5$ & $K^{-1}$ & $N^{-1} $ & $ a^{-1} A - a^{-2} K^{-1}$
\\
$6$ & $B^{-1}$ & $N^{-1} $ & $ a A - a^{2} B^{-1}$
\\
$7$ & $(K^\downarrow)^{-1}$ & $(N^\downarrow)^{-1} $ & $ a^{-1} A - a^{-2} (K^\downarrow)^{-1}$
\\
$8$ & $(B^\downarrow)^{-1}$ & $(N^\downarrow)^{-1} $ & $ a A - a^{2} (B^\downarrow)^{-1}$
\\
\end{tabular}}
\bigskip
\noindent In the above table we abbreviate
\begin{align}
\label{eq:NM}
&M= \frac{a K - a^{-1} B}{a - a^{-1}}, \qquad \qquad
N= \frac{a^{-1} K^{-1} - a B^{-1}}{a^{-1}-a},
\\
&M^\downarrow = \frac{a K^\downarrow- a^{-1} B^\downarrow}{a - a^{-1}}, \qquad \quad
N^\downarrow = \frac{a^{-1} (K^\downarrow)^{-1} - a (B^\downarrow)^{-1}}{a^{-1}-a}.
\label{eq:NMd}
\end{align}
\end{proposition}
\begin{proof} For examples 1, 2 the $q$-Weyl relations from Definition
\ref{def:et} are a routine consequence of (\ref{eq:KA}), (\ref{eq:KB}). They also follow from
 \cite[Lemmas~11.5, 12.8, Theorem~12.5]{bocktingTer}.
Examples 3, 4 are obtained from examples 1, 2 by replacing the decomposition $\lbrace V^*_i\rbrace_{i=0}^d$ by its inversion
$\lbrace V^*_{d-i}\rbrace_{i=0}^d$.
For examples 5, 6 the $q$-Weyl relations from Definition
\ref{def:et} are a routine consequence of (\ref{eq:KAa}), (\ref{eq:KBa}). They also follow from
 \cite[Lemmas~11.5, 12.16, Theorem~12.14]{bocktingTer}.
Examples 7, 8 are obtained from examples 5, 6 by replacing the decomposition $\lbrace V^*_i\rbrace_{i=0}^d$ by its inversion
$\lbrace V^*_{d-i}\rbrace_{i=0}^d$.
\end{proof}

\section{The main results}

\noindent We continue to discuss the $\mathcal O_q$-module $V$ from Section 2.
In this section we display some diagrams that illustrate how the eight equitable triples from
Proposition \ref{thm:xyz}
are related to 
 the eigenspace decompositions for $A$, $A^*$, $L^{\pm 1}(A^*)$ on $V$.
\medskip

\noindent
In the next result, we compare the diagrams in
Definition \ref{def:splitMap} and
Corollary
\ref{prop:LspiitMap}.
In order to make the comparison, we reflect the diagram
in Corollary \ref{prop:LspiitMap} about the horizonal line segment labelled $A$.

\newpage


\begin{theorem}
\label{thm:Example9} We have 
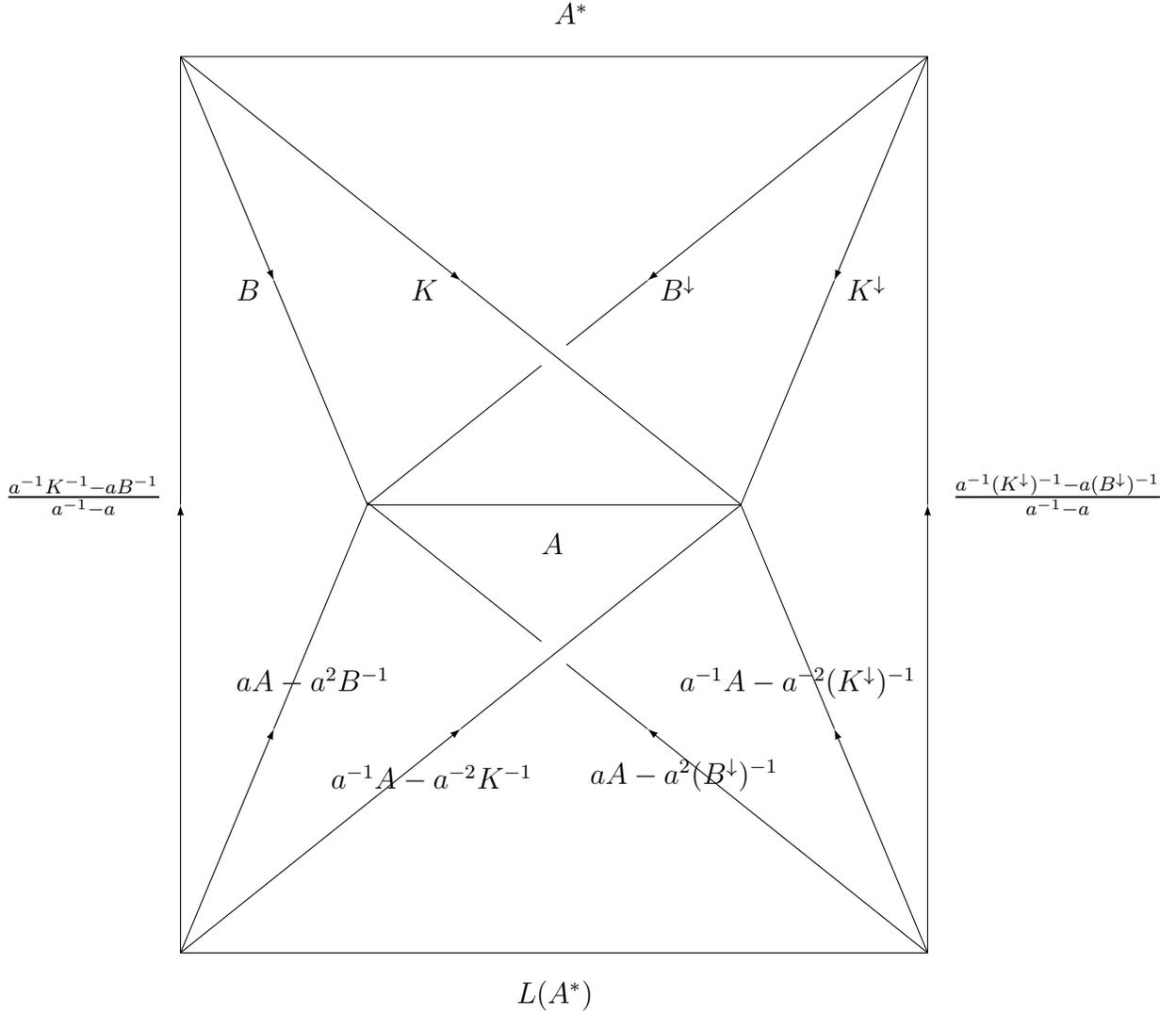
\begin{center}
\begin{picture}(0,230)
\put(-150,180){\line(2,0){300}}
\put(0,193){$A^*$}


\put(-150,180){\vector(75,-180){37.5}}
\put(-112.5,90){\line(75,-180){37.5}}
\put(-127.5,82){$B$}

\put(-150,180){\vector(225,-180){112.5}}
\put(-37.5,90){\line(225,-180){112.5}} 
\put(-57.5,82){$K$}

\put(150,180){\vector(-225,-180){112.5}}
\put(37.5,90){\line(-225,-180){32.5}}
\put(-75,0){\line(225,180){70}}
\put(42.5,82){$B^\downarrow$}

\put(150,180){\vector(-75,-180){37.5}}
\put(112.5,90){\line(-75,-180){37.5}}
\put(117.5,82){$K^\downarrow$}

\put(-150,-180){\vector(75,180){37.5}}
\put(-112.5,-90){\line(75,180){37.5}}

\put(-150,-180){\vector(225,180){112.5}}
\put(-37.5,-90){\line(225,180){112.5}}


\put(150,-180){\vector(-225,180){112.5}}
\put(37.5,-90){\line(-225,180){32.5}}
\put(-5,-55){\line(-225,180){70}}

\put(150,-180){\vector(-75,180){37.5}}
\put(112.5,-90){\line(-75,180){37.5}}

\put(-127.5,-75){$aA-a^{2}B^{-1}$}

\put(-89.5,-114){$a^{-1}A-a^{-2}K^{-1}$}


\put(14.5,-112){$aA-a^{2}(B^\downarrow)^{-1}$}

\put(50.5,-75){$a^{-1}A-a^{-2}(K^\downarrow)^{-1}$}

\put(-75,0){\line(1,0){150}}
\put(-5,-20){$A$}
\put(-150,-180){\line(1,0){300}}
\put(-15,-200){$L(A^*)$}

\put(-150,-180){\vector(0,1){180}}
\put(-150,0){\line(0,1){180}}
\put(-220,0){$\frac{a^{-1} K^{-1} - aB^{-1}}{a^{-1}-a}$}

\put(150,-180){\vector(0,1){180}}
\put(150,0){\line(0,1){180}}
\put(160,0){$\frac{a^{-1} (K^\downarrow)^{-1}- a(B^\downarrow)^{-1}}{a^{-1}-a}$}

\end{picture}
\end{center}
\end{theorem}

\vspace{7.5cm}
\begin{proof} Recall the abbreviations $N$, $N^\downarrow$ from
(\ref{eq:NM}), 
(\ref{eq:NMd}). 
 The diagram in the theorem statement is making some assertions about $N$ and $N^\downarrow$. We now verify these assertions, starting with $N$.
 By Lemma
 \ref{lem:fatInv}, $N$ is diagonalizable with eigenvalues $\lbrace q^{d-2i}\rbrace_{i=0}^d$.
 By Proposition
 \ref{thm:xyz} one finds  that the $q$-Weyl relation is satisfied by the pair $K^{-1}, N^{-1}$ and also the pair $N^{-1}$, $a^{-1}A-a^{-2}K^{-1}$.
  In the diagram of the theorem statement,
 the vertical line segment on the left represents
 the $N$-eigenspace decomposition of $V$. This line segment is properly attached in the diagram, due to
 Corollary \ref{cor:XY} 
 and our above findings.
 We have verified the assertions about $N$. The assertions about $N^\downarrow$ are verified in a similar manner.
\end{proof}

\newpage

\noindent  The following theorem contains a diagram. This diagram is obtained from the diagram in Theorem \ref{thm:Example9} by inverting the orientation for some of the edges.
\begin{theorem} 
\label{thm:Example6}  
For each oriented 3-cycle in the diagram below, the corresponding maps form an equitable triple.

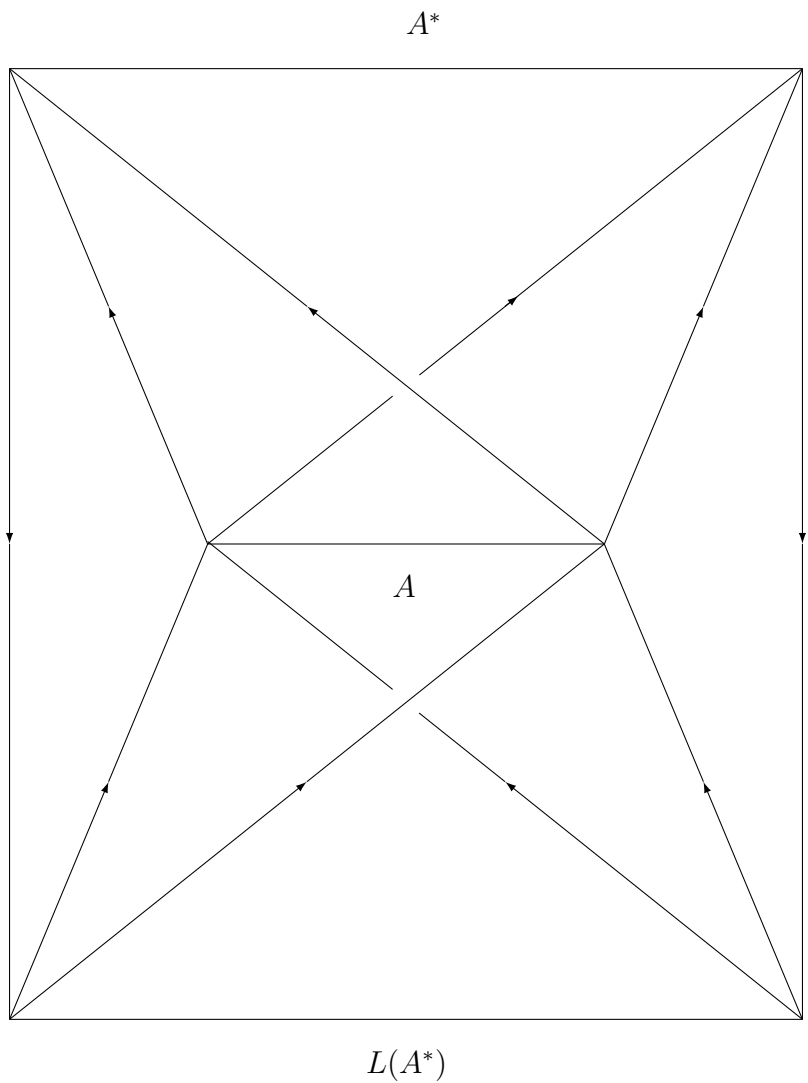
\begin{center}
\begin{picture}(0,230)
\put(-150,180){\line(2,0){300}}
\put(0,193){$A^*$}

\put(-75,0){\vector(-75,180){37.5}}
\put(-112.5,90){\line(-75,180){37.5}}

\put(75,0){\vector(-225,180){112.5}}
\put(-37.5,90){\line(-225,180){112.5}}

\put(-75,0){\line(225,180){70}}
\put(5,64){\vector(225,180){37.5}}
\put(37.5,90){\line(225,180){112.5}}

\put(75,0){\vector(75,180){37.5}}
\put(112.5,90){\line(75,180){37.5}}


\put(-150,-180){\vector(75,180){37.5}}
\put(-112.5,-90){\line(75,180){37.5}}

\put(-150,-180){\vector(225,180){112.5}}
\put(-37.5,-90){\line(225,180){112.5}}


\put(150,-180){\vector(-225,180){112.5}}
\put(37.5,-90){\line(-225,180){32.5}}
\put(-5,-55){\line(-225,180){70}}

\put(150,-180){\vector(-75,180){37.5}}
\put(112.5,-90){\line(-75,180){37.5}}

\put(-75,0){\line(1,0){150}}
\put(-5,-20){$A$}
\put(-150,-180){\line(1,0){300}}
\put(-15,-200){$L(A^*)$}

\put(-150,180){\vector(0,-1){180}}
\put(-150,0){\line(0,-1){180}}

\put(150,180){\vector(0,-1){180}}
\put(150,0){\line(0,-1){180}}

\end{picture}
\end{center}
\vspace{7.5cm}

\end{theorem}
\begin{proof} The above diagram contains four oriented 3-cycles. For each one, the corresponding maps form an equitable triple
by examples 5--8 in the table of  Proposition \ref{thm:xyz}. 
\end{proof}

\newpage
\noindent
In the next result, we compare the diagrams in
Definition \ref{def:splitMap} and
Corollary
\ref{prop:LspiitMapR}.
In order to make the comparison, we reflect the diagram
in Corollary \ref{prop:LspiitMapR} about the horizonal line segment labelled $A$.

\begin{theorem}
\label{thm:Example8}
We have

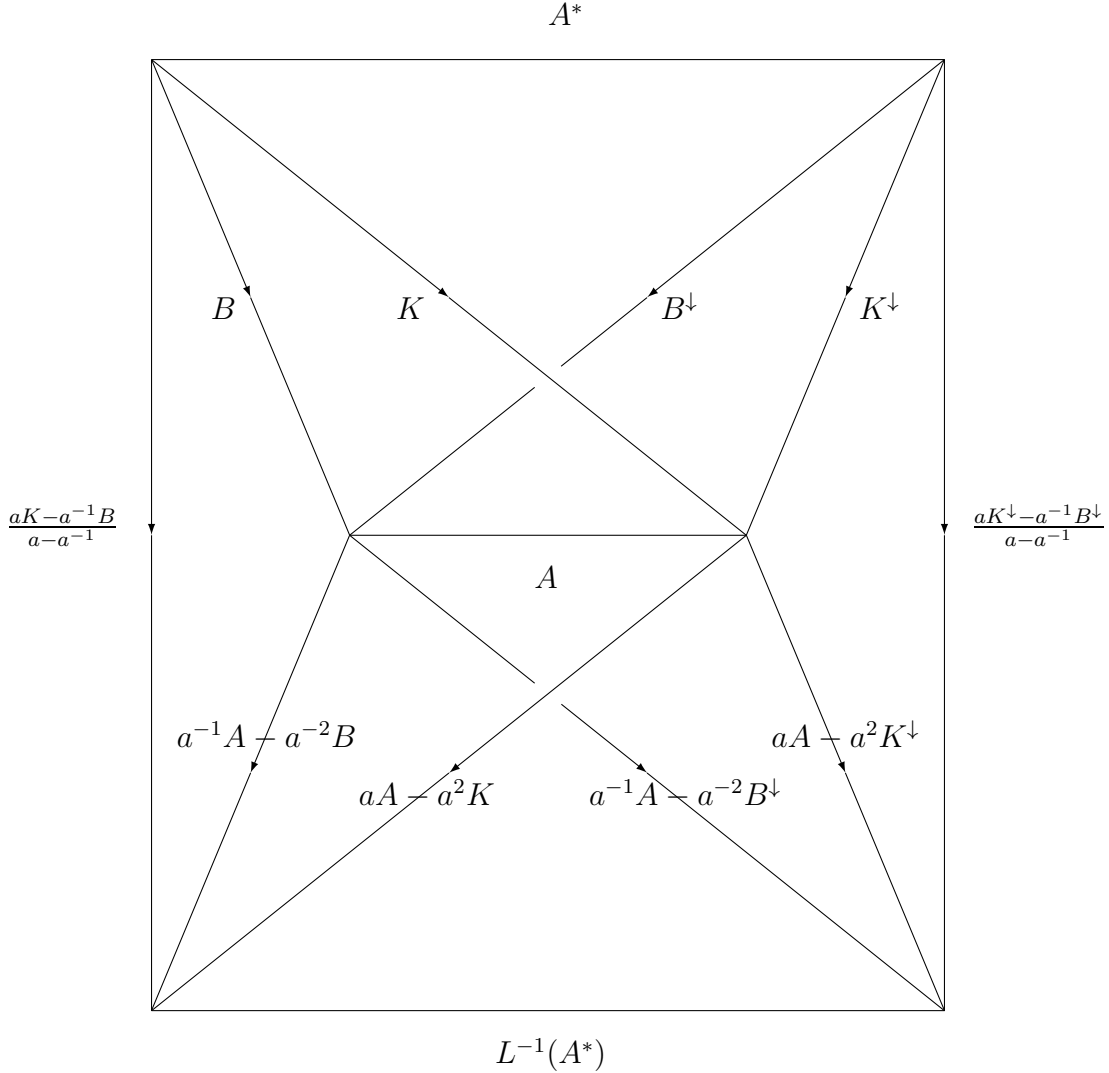
\begin{center}
\begin{picture}(0,230)
\put(-150,180){\line(2,0){300}}
\put(0,193){$A^*$}


\put(-150,180){\vector(75,-180){37.5}}
\put(-112.5,90){\line(75,-180){37.5}}
\put(-127.5,82){$B$}

\put(-150,180){\vector(225,-180){112.5}}
\put(-37.5,90){\line(225,-180){112.5}} 
\put(-57.5,82){$K$}

\put(150,180){\vector(-225,-180){112.5}}
\put(37.5,90){\line(-225,-180){32.5}}
\put(-75,0){\line(225,180){70}}
\put(42.5,82){$B^\downarrow$}

\put(150,180){\vector(-75,-180){37.5}}
\put(112.5,90){\line(-75,-180){37.5}}
\put(117.5,82){$K^\downarrow$}


\put(-75,0){\vector(-75,-180){37.5}}
\put(-112.5,-90){\line(-75,-180){37.5}}
\put(-140.5,-80){$a^{-1}A-a^{-2}B$}

\put(75,0){\vector(-225,-180){112.5}}
\put(-37.5,-90){\line(-225,-180){112.5}}
\put(-71.5,-102){$aA-a^{2}K$}


\put(-75,0){\line(225,-180){70}}
\put(5,-64){\vector(225,-180){32.5}}
\put(37.5,-90){\line(225,-180){112.5}}
\put(15.5,-102){$a^{-1}A-a^{-2}B^\downarrow$}

\put(75,0){\vector(75,-180){37.5}}
\put(112.5,-90){\line(75,-180){37.5}}
\put(84.5,-80){$aA-a^{2}K^\downarrow$}

\put(-75,0){\line(1,0){150}}
\put(-5,-20){$A$}
\put(-150,-180){\line(1,0){300}}
\put(-20,-200){$L^{-1}(A^*)$}

\put(-150,180){\vector(0,-1){180}}
\put(-150,0){\line(0,-1){180}}
\put(-205,0){$\frac{a K - a^{-1}B}{a-a^{-1}}$}

\put(150,180){\vector(0,-1){180}}
\put(150,0){\line(0,-1){180}}
\put(160,0){$\frac{a K^\downarrow - a^{-1}B^\downarrow}{a-a^{-1}}$}
\end{picture}
\end{center}
\end{theorem}
\vspace{7.5cm}
\begin{proof} Recall the abbreviations $M$, $M^\downarrow$ from
(\ref{eq:NM}), 
(\ref{eq:NMd}). 
 The diagram in the theorem statement is making some assertions about $M$ and $M^\downarrow$. We now verify these assertions, starting with $M$.
 By Lemma
 \ref{lem:fatInv}, $M$ is diagonalizable with eigenvalues $\lbrace q^{d-2i}\rbrace_{i=0}^d$.
 By Proposition
 \ref{thm:xyz} one finds  that the $q$-Weyl relation is satisfied by  the pair $a A-a^{2}K$, $M^{-1}$ and also the pair $M^{-1}$, $K$.
  In the diagram of the theorem statement,
 the vertical line segment on the left represents
 the $M$-eigenspace decomposition of $V$. This line segment is properly attached in the diagram, due to
 Corollary \ref{cor:XY} 
 and our above findings.
 We have verified the assertions about $M$. The assertions about $M^\downarrow$ are verified in a similar manner.
\end{proof}

\newpage
\noindent  The following theorem contains a diagram. This diagram is obtained from the diagram in 
Theorem \ref{thm:Example8}
 by inverting the orientation for some of the edges.

\begin{theorem} 
\label{thm:Example7}
For each oriented 3-cycle in the diagram below, the corresponding maps form an equitable triple.
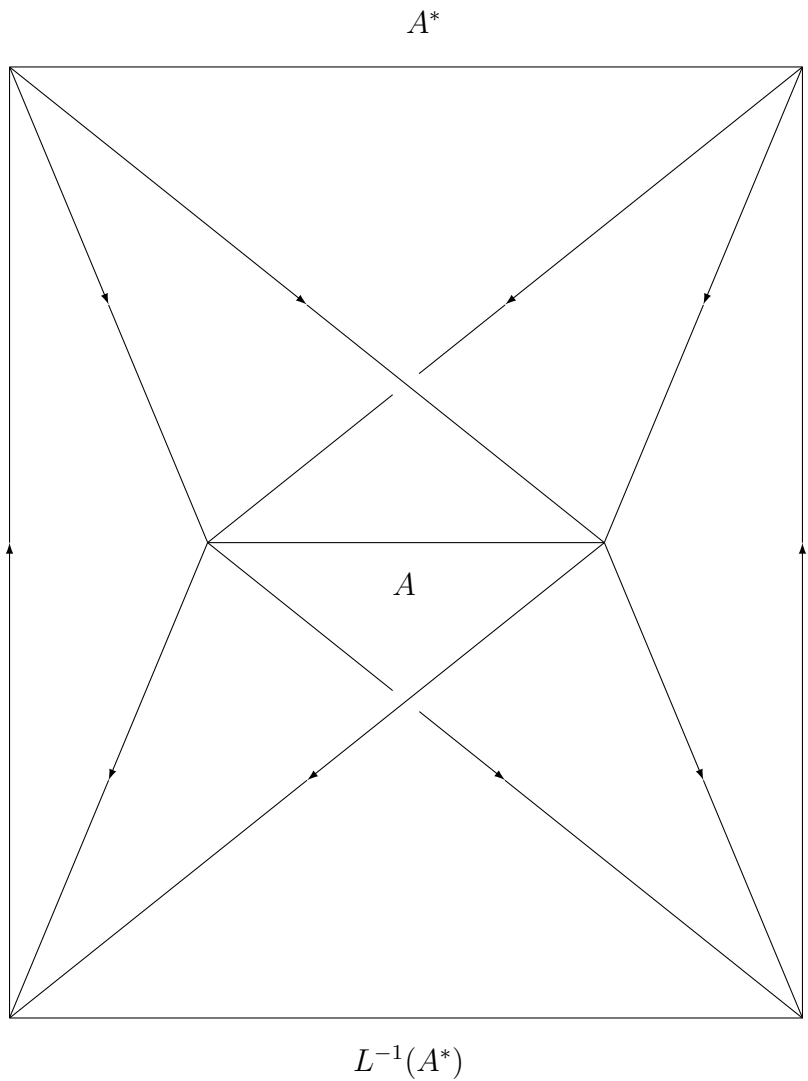
\begin{center}
\begin{picture}(0,230)
\put(-150,180){\line(2,0){300}}
\put(0,193){$A^*$}


\put(-150,180){\vector(75,-180){37.5}}
\put(-112.5,90){\line(75,-180){37.5}}

\put(-150,180){\vector(225,-180){112.5}}
\put(-37.5,90){\line(225,-180){112.5}} 

\put(150,180){\vector(-225,-180){112.5}}
\put(37.5,90){\line(-225,-180){32.5}}
\put(-75,0){\line(225,180){70}}

\put(150,180){\vector(-75,-180){37.5}}
\put(112.5,90){\line(-75,-180){37.5}}


\put(-75,0){\vector(-75,-180){37.5}}
\put(-112.5,-90){\line(-75,-180){37.5}}

\put(75,0){\vector(-225,-180){112.5}}
\put(-37.5,-90){\line(-225,-180){112.5}}


\put(-75,0){\line(225,-180){70}}
\put(5,-64){\vector(225,-180){32.5}}
\put(37.5,-90){\line(225,-180){112.5}}

\put(75,0){\vector(75,-180){37.5}}
\put(112.5,-90){\line(75,-180){37.5}}

\put(-75,0){\line(1,0){150}}
\put(-5,-20){$A$}
\put(-150,-180){\line(1,0){300}}
\put(-20,-200){$L^{-1}(A^*)$}

\put(-150,-180){\vector(0,1){180}}
\put(-150,0){\line(0,1){180}}

\put(150,-180){\vector(0,1){180}}
\put(150,0){\line(0,1){180}}

\end{picture}
\end{center}
\end{theorem}
\vspace{7.5cm}

\begin{proof} The above diagram contains four oriented 3-cycles. For each one, the corresponding maps form an equitable triple
by examples 1--4 in the table of Proposition \ref{thm:xyz}. 
\end{proof}
\section{Acknowledgement} The author thanks Pascal Baseilhac and Edward Hanson, for giving the paper a close reading and offering valuable comments.
The author thanks Kazumasa Nomura, for explaining how to create the diagrams and double checking the equations of the
paper by computer.
\newpage

\bigskip

\noindent Paul Terwilliger \hfil\break
\noindent Department of Mathematics \hfil\break
\noindent University of Wisconsin \hfil\break
\noindent 480 Lincoln Drive \hfil\break
\noindent Madison, WI 53706-1388 USA \hfil\break
\noindent email: {\tt terwilli@math.wisc.edu }\hfil\break

\end{document}